\documentclass[12pt]{amsart}
\usepackage{a4wide}
\usepackage{enumerate}
\usepackage{graphicx}
\usepackage{color}
\usepackage{amsmath}
\usepackage{amssymb}
\usepackage{comment}
\usepackage{soul}
\allowdisplaybreaks

\let\pa\partial  
\let\del\partial  
\let\na\nabla  
\let\eps\varepsilon  

\newcommand{\newln}{\\&\quad\quad{}}
\newcommand{\R}{{\mathbb R}}

\newcommand{\T}{{\mathbb T}}
\newcommand{\F}{{\mathcal F}}

\newcommand{\N}[1]{\left|#1\right|}
\newcommand{\NN}[1]{\left\|#1\right\|}

\newcommand{\velo}{u}

\newcounter{hyp}
\newtheorem{theorem}{Theorem}   
\newtheorem{lemma}[theorem]{Lemma}   
\newtheorem{proposition}[theorem]{Proposition}   
\newtheorem{remark}[theorem]{Remark}   
  
\newtheorem{definition}{Definition}  
\newtheorem{example}{Example} 

\begin{document}  

\title[Semiconductor Boltzmann-Dirac-Benney equation]{Global analytic solutions of the Semiconductor Boltzmann-Dirac-Benney equation with relaxation time approximation}

\author{MARCEL BRAUKHOFF}

\address{Institute for Analysis and Scientific Computing, Vienna University of  
	Technology, Wiedner Hauptstra\ss e 8--10, 1040 Wien, Austria  \\
	marcel.braukhoff@asc.tuwien.ac.at}

\date{\today}

\thanks{The authors acknowledge partial support from   
the Austrian Science Fund (FWF), grants P27352, P30000, and W1245.} 

\begin{abstract}
	
The global existence of a solution of the semiconductor Boltzmann-Dirac-Benney equation 
	\[
	\del_t f + \nabla\epsilon(p)\cdot\nabla_x f - \nabla \rho_f(x,t)\cdot\nabla_p f =   \frac{\F_\lambda(p)-f}\tau, \quad x\in\mathbb{R}^d,\ p\in B, \ t>0
	\]
 is shown for small  $\tau>0$ assuming that the initial data are analytic and sufficiently close to $\F_\lambda$. This system contains an interaction potential $\rho_f(x,t):=\int_{B}f(x,p,t)dp$ being significantly more singular than the Coulomb potential, which causes major structural difficulties in the analysis. The semiconductor Boltzmann-Dirac-Benney equation is a model for ultracold atoms trapped in an optical lattice. Hence, the dispersion relation is given by $\epsilon(p) = -\sum_{i=1}^d$  $\cos(2\pi p_i)$, $p\in B=\mathbb{T}^d$ due to the optical lattice and the Fermi-Dirac distribution $\F_\lambda(p)=1/(1+\exp(-\lambda_0-\lambda_1\eps(p)))$ describes the equilibrium of ultracold fermionic clouds.  
 
 This equation is closely related to the Vlasov-Dirac-Benney equation with $\epsilon(p)=\frac{p^2}2$, $p\in B=\mathbb R^d$ and r.h.s$.=0$, where the existence of a global solution is still an open problem.  So far, only local existence and ill-posedness results were found for theses systems.
 
The key technique is based of the ideas of Mouhot and Villani by using Gevrey-type norms which vary over time. The global existence result for small initial data is also shown for a far more general setting, namely
 \[\del_t f + Lf=Q(f),\]
 where $L$ is a generator of an $C^0$-group with $\|e^{tL}\|\leq Ce^{\omega t}$ for all $t\in\R$ and $\omega>0$ and, where further additional analytic properties of $L$ and $Q$ are assumed.
%
%
%
\end{abstract}

 

\maketitle


\section{Introduction}
The semiconductor Boltzmann-Dirac-Benney equation is a model describing ultracold atoms in an optical lattice. An optical lattice is a spatially periodic structure that is formed by interfering
optical laser beams. The interference produces an optical standing wave that may trap
neutral atoms \cite{Blo05}. The underlying experiment has been proved to be a powerful tool to study physical phenomena that occur in sold state materials. Simply speaking, a solid crystal consists of ions and electrons. Because of the mass difference, the electrons in average move much faster than the ions in a semi-classical picture. Therefore, from a modeling point of view, one may assume that the positions of the ions are fixed and form a regular periodic structure. However, comparing the theory to the experiment, one faces certain difficulties as impurities lead to defects in the periodic structure.

The experiment of ultracold atoms in an optical lattice can be considered as a physical toy-model for solid state materials. The ultracold atoms represent the electrons and the optical lattice mimics the periodic structure of the ions. The advantage of the optical lattice is the absence of impurities. Thus, one expects a better accordance of the experiment with the theory. Moreover, the dynamics of the ultracold atoms, i.e.~at a temperature of magnitude of some
nanokelvin, can be followed on the time scale of milliseconds.
This facilitates the study physical phenomena in an optical lattice being difficult to observe in solid crystals. Furthermore, they are promising candidates to realize quantum information processors \cite{Jak04} and extremely precise atomic clocks \cite{ADH15}.

 The main difference consists of the use of uncharged atoms, whereas electrons are negatively charged. Ultracold fermions may be described with a Fermi-Hubbard model with a Hamiltonian that is a result of the lattice potential created by interfering laser beams and short-ranged
 collisions \cite{DGH15}. They assume that the ultracold atoms interact only with their nearest neighbors. For more details see \cite{Man12}. 
 
 In this article we are focusing on a semi-classical picture which is able to model qualitatively the observed cloud shapes \cite{SHR12}. The effective dynamics are modeled by a Boltzmann transport equation describing the microscopic particle density $f=f(x,p,t)$, where $x\in\mathbb R^d$ is the position, $p\in B$ the momentum and $t\geq0$ the time. In the prototype case, we assume that the potential forms a simple cubic lattice. Here, we identify the first Brillouin zone $B:=[0,2\pi)^d\subset \R^d$ with the torus $\T^d$. The band energy $\eps(p)$ is given by the periodic dispersion relation
\begin{equation*}
\eps(p) = -2\eps_0\sum_{i=1}^d \cos(2\pi p_i), \quad p\in\T^d.
\end{equation*}
The constant $\eps_0$ is a measure for the tunneling rate of a particle from one
lattice site to a neighboring one.  This dispersion relation also occurs as an approximation for the lowest energy band in semiconductors (see \cite{AsMe77}). Let $\rho_f:=\int_{\mathbb T^d}fdp$ be the macroscopic particle density. The interaction potential is given by $V_f=-U\rho_f$, where  $U>0$ models the strength of the on-site interaction between spin-up and spin-down components \cite{SHR12}. 

Finally, the semiconductor Boltzmann-Dirac-Benney equation is given by
\begin{equation}\label{1.be}
\pa_t f + \nabla\eps(p)\cdot\na_x f - U\na_x \rho_f\cdot\na_p f = Q(f),
\end{equation}
where $Q(f)$ is a collision operator. There are several choices for the collision operator. The natural choice of the collision operator is a two particle collision operator neglecting the three or more particle scattering 
\begin{multline*}
Q_{ee}(g)(p):= \sum_{G\in2\pi \mathbb Z^d}\int\limits_{\mathbb T^d}\int\limits_{\substack{p_{\mathrm{tot}}(\mathbf{p})=G\\\eps_{\mathrm{tot}}(\mathbf{p})=0 }}\!\!Z(\mathbf{p})\bigg(
g(p)g(p')(1-\eta g(p''))(1-\eta g(p'''))\\-g(p'')g(p''')(1-\eta g(p))(1-\eta g(p'))
\bigg)\frac{d\mathcal H_{p''}^{d-1}}{\N{\nabla_{p''}\eps_{\mathrm{tot}}(\mathbf{p})}}dp'.
\end{multline*}
for some $\eta\geq0$, where $\mathbf{p}=(p,p',p'',p''')$ and $\mathcal H^{d-1}_{p''}$ denotes the $d -1$ dimensional Hausdorff measure w.r.t.~$p''$.
The function $Z(\mathbf{p})$ models the probability of a scattering event from state $(p,p_1)$ to the state $(p_2,p_3)$. Moreover, the total change of momentum and energy are denoted by 
\[p_{\mathrm{tot}}(\mathbf{p}):= p+p'-p''-p'''\quad\mbox{and}\quad \eps_{\mathrm{tot}}(\mathbf{p})=\eps(p)+\eps(p')-\eps(p'')-\eps(p'''), \]
respectively. The sum over $G$ runs over all reciprocal lattice vectors $G\in 2\pi \mathbb Z^d$. Note that in fact only finite summands contribute to the sum since $p_{\mathrm{tot}}$ is bounded. This scattering operator is also well-known as the electron-electron scattering operator \cite{BeDe96}.

Comparing the semiconductor Boltzmann-Dirac-Benney equation to the semiconductor Boltzmann
equation with Coulomb interaction, there are two major differences. First, the band energy $\epsilon$ is a bounded function in contrast to the parabolic band approximation $\eps(p)=\frac12|p|^2$, which is usually assumed \cite{Jue09}. Second, the potential $V_f$ is proportional to the macroscopic particle density $\rho_f=\int_{\T^d}fdp$. In semiconductor physics, the interaction potential $\Phi_f$ between the electrons is often modeled by the Coulomb potential \cite{Jue09}. Hence, $\Phi_f$ is determined self-consistently from the Poisson equation $-\Delta \Phi_f=\rho_f$ and therefore much more regular that $V_f$.

\subsection*{Fermi-Dirac distribution}
Due to the complexity of the two particle scattering operator, the analysis of \eqref{1.be} with $Q=Q_{ee}$ is very difficult. Therefore, we search for a less complicated physical approximation of $Q_{ee}$. In \cite{Jue09}, J\"ungel proves in Proposition 4.6 that the zero set of $Q_{ee}$ consists of Fermi-Dirac distribution functions, i.e. it holds formally that $Q_{ee}(g)=0$ if and only if there exists a $\lambda=(\lambda_0,\lambda_1)\in\R^2$ with
\begin{equation*}
g(p)=\F_\lambda(p):= \frac{1}{\eta+e^{-\lambda_0-\lambda_1\eps(p)}}.
\end{equation*}
Hence, $\F_\lambda$ annihilates the collision operator and can be seen as an equilibrium distribution. For $\eta=1$, we obtain the Fermi-Dirac distribution, while for
$\eta=0$, $\F_\lambda$ equals the Maxwell-Boltzmann distribution. The parameter $\lambda_0,\lambda_1$ are sometimes called entropy parameters, where physically $-\lambda_1$ equals the inverse temperature and $-\lambda_0/\lambda_1$ the chemical potential.

Note that we have assumed a bounded band energy. This implies that the equilibrium $\F_\lambda$ is integrable w.r.t.~$p$ even if $\lambda_1>0$, which means that the absolute temperature may be negative.
In fact, negative absolute temperature can be realized in experiments with
ultracold atoms \cite{RMR10}. Negative temperatures occur in
equilibrated (quantum) systems that are characterized by an inverted population of
energy states. The thermodynamical implications of negative temperatures 
are discussed in \cite{Ram56}.
 
\subsection*{Relaxation time approximation}
 The idea of the relaxation time approximation is to assume that the collision operator drives the solution into the equilibrium. We define
\begin{equation*}
Q(g)(p):=\frac{\F_\lambda(p)-g(p)}{\tau}
\end{equation*}
for some $\lambda\in \R^2$, $\tau>0$ and $g=g(p)$ being a heuristic approximation of $Q_{ee}$ \cite{AsMe77}. The parameter $\tau$ is called the relaxation time and represents the average time between two scattering events. Since $\F_\lambda$ is a fixed function, the relaxation time approximation collision operator neither conserves the local particle nor the local energy. The simplest version of the relaxation time approximation is to assume that $\lambda_1$ vanishes. Then, $\F_{\lambda_0,0}$ equals a constant $\overline{\rho}\in[0,1/\eta]$.

\subsection*{Known results}
In a previous paper \cite{Bra18}, the semiconductor Boltzmann-Dirac-Benney equation is investigated with a BGK-type collision operator
\begin{equation}\label{eq.Qf.BGK}
Q_{BGK}(f) = \frac{\rho_f(1-\eta \rho_f)}{\tau}(F_f-f),
\end{equation}
where $\tau>0$ is the relaxation time
and $F_f$ is determined by 
$$
F_f(x,p,t) = \frac{1}{\eta + e^{-\bar\lambda_0(x,t)-\bar\lambda_1(x,t)\eps(p)}},
\quad x\in\R^d,\ p\in\T^d,\ t>0,
$$
where $(\bar\lambda_0,\bar\lambda_1)$ are the Lagrange multipliers resulting from the local 
mass and energy conservation constraints, i.e.
$$
\int_{\T^d}(F_f-f)dp=0,\quad\int_{\T^d}(F_f-f)\eps(p)dp=0.
$$
In \cite{Bra18}, it is shown that \eqref{1.be} with $Q=Q_{BGK}$ is ill-posed in the following sense.

	Let $k\in\mathbb N$, $\theta>0$ and $\gamma>0$, $U\neq0$. There exist ${\lambda}\in\R^2$ and a time $\tau>0$ and such that there exist solutions $f_\delta:\R^d_x\times\T^d_p\times[0,\tau]\to [1,\eta^{-1}]$ of \eqref{1.be} with $Q=Q_{BGK}$ such that
	
	\begin{equation*}
	\lim_{\delta\to0}\frac{\NN{f_\delta(\cdot,\cdot,t)-\F_\lambda}_{L^1(B_\delta(x,p))}}{\NN{f_\delta(\cdot,\cdot,0)-\F_\lambda}^\theta_{W^{1,k}(\R^3\times\T^3)}}=\infty\quad\mbox{for all }x\in\R^d,p\in\T^d,t\in(0,\tau).
	\end{equation*}
	 A sufficient condition for the critical $\bar{\lambda}$ is given in \cite{Bra18} by
	\begin{equation*}
	1<U \lambda_1\int_{\T^d}\F_\lambda(p)(1-\eta \F_\lambda(p))dp.
	\end{equation*}
%
 
This result reflects the theory of the Vlasov-Dirac-Benney equation with is the counterpart of the semiconductor Boltzmann-Dirac-Benney equation for free particle without collisions, i.e. with $\epsilon(p)=\frac12|p|^2$ and $Q(f)=0$.

The Vlasov-Dirac-Benney equation is therefore given by
\begin{equation}\label{eq: intro: Vlasov3}
\begin{aligned}
\pa_t f+ p\cdot\na_x f-\na\rho_f(x,t)\cdot\na_p f &=0
\end{aligned}
\end{equation}
for $x\in\R^d,p\in\R^d$ and $t >0$. In spatial dimension one, this equation can be used to describe the density of a fusion plasma in a strong magnetic field in direction of the field \cite{BaNo12}.

The Vlasov-Dirac-Benney equation is a limit of a scaled non-linear Schr\"odinger equation \cite{BaBe16}. 
Comparing the standard Vlasov-Poisson equation, we see that the interaction potential $\Phi_f:=-\frac{1}{|x|}*\rho_f$ is long ranged by means of that the support  of the kernel $1/|x|$ is the whole space. The interaction potential of the Vlasov-Dirac-Benney equation can be rewritten using the $\delta$ distribution as $V_f:=-U\rho_f=-U\delta_0*\rho_f$. Therefore $V_f$ is called  a short-ranged Dirac potential, which motivated the ``Dirac" in the name of the Vlasov-Dirac-Benney equation \cite{BaBe13}. The name Benney is due to its relation to the Benney equation in dimension one (for details see \cite{BaBe13}). Moreover, the Vlasov-Dirac-Benney equation can also be derived by a quasi-neutral limit of the Vlasov-Poisson equation \cite{HaRo16}.

The Vlasov-Dirac-Benney equation first appeared in \cite{JaNo11}, where  only local in time solvability was shown for analytic initial data in spatial dimension one. In \cite{BaBe13}, Bardos and Besse show that this system is not locally weakly $(H^m-H^1)$ well-posed in the sense of Hadamard. Moreover, the Vlasov-Dirac-Benney equation is actually ill-posed in $d=3$, requiring that the spatial domain is restricted to the $3$-dimensional torus $\T^3$ \cite{HaNg15}: the flow of solutions does not belong to $C^\alpha(H^{s,m}(\R^3\times\T^3),L^2(\R^3\times\T^3))$ for any $s\geq0,\alpha\in(0,1]$ and $m\in\mathbb N_0$. Here, $H^{s,m}(\R^3\times\T^3)$ denotes the weighted Sobolev space of order $s$ with weight $(x,\velo)\mapsto \langle\velo\rangle^m:=(1+\N{\velo}^2)^{m/2}$. 
More precisely, \cite{HaNg15} provides a stationary solution $\mu=\mu(\velo)$ of \eqref{eq: intro: Vlasov3} and a family of solutions $(f_\varepsilon)_{\varepsilon>0}$, times $t_\varepsilon= O(\varepsilon\N{\log\varepsilon})$ and $(x_0,\velo_0)\in\T^3\times\R^3$ such that

\begin{equation}\label{kwan}
\lim_{\varepsilon\to0}\frac{\NN{f_\varepsilon-\mu}_{L^2([0,t_\varepsilon]\times B_\varepsilon(x_0)\times B_\epsilon(\velo_0))}}{\NN{\langle\velo\rangle^m(f_\varepsilon|_{t=0}-\mu)}^\alpha_{H^s(\T^3_x\times\R^3_\velo)}}=\infty,
\end{equation}
where $B_\varepsilon(x_0)$ denotes the ball with radius $\varepsilon$ centered at $x_0$. 

These results show the main difference between the well-posed Vlasov-Poisson equation and the Vlasov-Dirac-Benney equation. 

 In \cite{HaRo16}, Han-Kwan and Rousset  consider the quasi-neutral limit of the Vlasov-Poisson equation. By proving uniform estimates on the solution of the scaled Vlasov-Poisson equation the show that the scaled solution converges to a unique local solution $f\in C([0,T], H^{2m-1,2r}(\R^3\times\T^3))$ of the Vlasov-Dirac-Benney equation. For this, they require that the initial data $f_0\in  H^{2m,2r}(\R^3\times\T^3)$ satisfies the Penrose stability condition 

\begin{equation*}
\inf_{x\in\mathbb T^d}\inf_{(\gamma,\tau,\eta)\in(0,\infty)\times\mathbb R\times\mathbb R^d\setminus\{0\}}\bigg|1-\int_0^\infty e^{-(\gamma+i\tau)s}\frac{i\eta}{1+|\eta|^2}\cdot\mathcal (F_v\nabla vf)(x,\eta s)ds\bigg|>0,
\end{equation*}
where $\mathcal{F}_v$ denotes the Fourier Transform in $v$.

 Note that the Vlasov-Dirac-Benney equation embeds into a larger class of ill-posed equation: Han-Kwan and Nguyen write Eq.~\eqref{eq: intro: Vlasov3} as a particular case of 
 \[\pa_t f+ L f= Q(f,f),\quad x\in \T^d, z\in\Omega\]
 in which $L$ (resp.~$Q$) is a is a linear (resp.~bilinear) integro-differential operator in $(x,z)$ and $\Omega$ is a open subset of $\R^k$ \cite{HaNg15}. They also state a version of \eqref{kwan} for the generalized setting by using the techniques of \cite{Met05}.
 
 \subsection{Main results}
 
 The semiconductor Boltzmann-Dirac-Benney equation and the Vla\-sov-Dirac-Benney equation have only been treated locally so far. A global existence result is still missing. The aim of this article is to show that the semiconductor Boltzmann-Dirac-Benney equation admits global solutions if we use the relaxation time approximation, namely
 \begin{equation}\label{2.be}
 \pa_t f + \nabla\eps(p)\cdot\na_x f - U\na_x \rho_f\cdot\na_p f = \frac{\F_\lambda(p)-f}{\tau}.
 \end{equation}
  For this we require analytic initial data being close to the Fermi-Dirac distribution \[\mathcal F_\lambda(p)=\frac1{\eta+e^{-\lambda_0-\lambda_1\eps(p)}},\quad p\in \mathbb T^d.\]
   This is due to the singular short ranged potential.
\begin{theorem}\label{thm1}
	Let $\lambda\in\mathbb R^2_+$, $U>0$ and $k\in\mathbb N$.	Then there exist $\tau_0,\varepsilon,\nu>0$ such that if $f_0\in S(\mathbb R^d\times\mathbb T^d)$ satisfies
	\begin{equation}\label{con.on.f_0.7}
	\sum_{\alpha,\beta\in\mathbb N_0}\frac{\nu^{{\alpha+\beta}}}{\alpha!\beta!}\|\partial_x^\alpha\partial_p^\beta(f_0-\mathcal F_\lambda)\|_{L^2(\mathbb R^d\times\mathbb T^d)}\leq \varepsilon 
	\end{equation} 
	 then for all $\tau\in(0,\tau_0)$, \eqref{2.be} has a unique global analytic solution with $f|_{t=0}=f_0$ satisfying
	 	\begin{equation*}
	 	\|f(t)-\mathcal F_\lambda\|_{H^k_xL^2_p}\leq Ce^{-(\frac1\tau-\frac{1}{\tau_0})t}\quad\mbox{for all }t\geq0
	 	\end{equation*}
	 	for some $C>0$. 	Moreover, for all $f_0,\tilde f_0\in S(\mathbb R^d\times\mathbb T^d)$ satisfying \eqref{con.on.f_0.7}, we have 
	 	\begin{equation*}
	 	\|f(t)-\tilde f(t)\|_{H^k_xL^2_p}\leq Ce^{-(\frac1\tau-\frac{1}{\tau_0})t}\sum_{\alpha,\beta\in\mathbb N_0}\frac{\nu^{{\alpha+\beta}}}{\alpha!\beta!}\|\partial_x^\alpha\partial_p^\beta(f_0- \tilde f_0)\|_{L^2(\mathbb R^d\times\mathbb T^d)}\quad\mbox{for all }t\geq0
	 	\end{equation*}
	 	for some $C>0$, where $f,\tilde f$ are the solution of \eqref{2.be} with $f(0)=f_0$ and $\tilde f(0)=\tilde f_0$, respectively.
\end{theorem}
We can also improve this result and obtain a better estimate for the solution $f$. For this, however, we require different spaces.
	\begin{definition}
		Let $\mathcal S(\mathbb R^d\times\mathbb T^d)$ and $C^\infty_b(\mathbb R^d\times\mathbb T^d)$ be the  Schwartz space and the space of bounded smooth functions, respectively.
	\begin{itemize}
		\item For  $\lambda\in \mathbb R^2_+:=\{(x,y)\in\mathbb R^2:y\geq0\}$, let $k\in\mathbb N$,  $k>\frac{d}{2}$. We define $Y:=\mathcal S(\mathbb R^d\times\mathbb T^d)$ and $X:= H^k_xL^2_p(\mathbb R^d_x\times\mathbb T^d_p)$ equipped with the scalar product
		\begin{equation*}
		\langle f,g\rangle_X:=\sum_{|\alpha|\leq k}\langle \partial_x^\alpha f,\partial_x^\alpha g\rangle_0,
		\end{equation*}
		where 
		\begin{equation*}
		\langle f,g \rangle_0:=
		\int_{\mathbb R^d}\int_{\mathbb T^d} f(x,p) g(x,p)\frac{dpdx}{\mathcal F_\lambda(p)(1-\eta \mathcal F_\lambda(p))}+U\lambda_1\int_{\mathbb R^d}\rho_f\rho_g dx
		\end{equation*}
		and $\rho_f(x):=\int_{\mathbb T^d}f(x,p)dp$.
		\item For $\lambda_1=0$, we can alternatively define $Y:= C^\infty_b(\mathbb R^d\times\mathbb T^d)$ and $X:= C^0_b(\mathbb R^d\times\mathbb T^d)$.
	\end{itemize} 
\end{definition}
\begin{definition}\label{def2}
	Let $\nu\in [0,\infty)$. We define
	\begin{equation*}
	\NN{\phi}_{Y^{\nu}_{t}}:= \sum_{|a+b|\leq1}\sum_{\alpha,\beta\in\mathbb N_0^n} \frac{\nu^{|\alpha+\beta|}}{\alpha!\beta!}\NN{e^{tL}\partial_x^{\alpha+a}\partial_p^{\beta+b}e^{-tL}\phi}_{X}
	\end{equation*}
	for $\phi\in Y$ and $t\in \R$, where $e^{tL}$ is generated by
	\begin{equation*}
	Lf(x,p):=\nabla\eps(p)\cdot\nabla_xf(x,p)-U\nabla_x\int_{\mathbb T^d}f(x,p')dp'\cdot\nabla_p\mathcal F_\lambda(p).
	\end{equation*}
	We show in Lemma \ref{lem.L.generator} that this is well-defined.
\end{definition}
\begin{theorem}\label{thm.diracbenny1}
	Let $\lambda\in\mathbb R^2_+$, $U>0$.	Then there exist $\tau_0,\varepsilon,\nu_0>0$ such that if 
	\begin{equation}\label{con.on.u_0.6}
	\|f_0-\mathcal F_\lambda\|_{Y_0^{\nu}}\leq \varepsilon \nu
	\end{equation} 
	for some $\nu\leq \nu_0$, then \eqref{2.be} has a unique global analytic solution $f$ with $f|_{t=0}=f_0$ for all $\tau\in(0,\tau_0)$. The solution satisfies 
	\begin{equation*}
	\|f(t)-\mathcal F_\lambda\|_{Y_t^{\nu \exp({-\frac t{\tau}})}}\leq 2\varepsilon\nu e^{-(\frac1\tau-\frac{1}{\tau_0}) t}\quad\mbox{for all }t\geq0.
	\end{equation*}
	 Moreover, for all $f_0,\tilde f_0\in Y$ satisfying \eqref{con.on.u_0.6}, we have 
	\begin{equation*}
	\|f(t)-\tilde f(t)\|_{Y_t^{\nu \exp({-\frac t{\tau}})}}\leq 2e^{-(\frac1\tau-\frac{1}{\tau_0}) t}\|f_0-\tilde f_0\|_{Y_0^{\nu}}\quad\mbox{for all }t\geq0
	\end{equation*}
	where $f,\tilde f$ are the solution of \eqref{2.be} with $f(0)=f_0$ and $\tilde f(0)=\tilde f_0$, respectively.
\end{theorem}

As in \cite{HaNg15}, we can generalize these results to a more abstract setting. Let $X$ be a Banach space and $Y\subset X$ be dense. Moreover, let $A=(A_1,\ldots,A_n):D(A)\subset X\to X^n$ be a linear operator with $Y\in D(A)$ such that $A(Y)\subset Y^n$ and $[A_i,A_j]=0$ for all $i,j=1,\ldots,n$. For $x_0\in D(A)$, we consider the non-linear Cauchy-problem
\begin{equation}\label{d_tx=F}
\del_t x=F(x),\quad\mbox{with}\quad x(0)=x_0,
\end{equation}
where $F:D(A)\to X$ satisfies the following conditions:
\begin{enumerate}
	\item[(H1)] There exists an $\bar x\in D(A)$ with $F(\bar x)=0$
	\item[(H2)] $F$ is G\^ateaux differentiable at $\bar x$ and $Lu:=DF(\bar x)u$ fulfills
	\begin{enumerate}
		\item[(H2a)] $L:D(L)\subset X\to X$ is a generator of a $C_0$ group $(e^{tL})_{t\in\R}$ with $Y\subset D(L)$ and $L(Y)\subset Y$ as well as
		\begin{equation*}
		\NN{e^{tL}}_X\leq C_L e^{\omega t}\quad\mbox{for all }t\in\R
		\end{equation*}
		and some $C_L\geq1$ and $\omega>0$. 
		\item[(H2b)] 	For $0\neq\alpha\in\mathbb N_0^n$, we define $L_{0}:=L$ and
		\begin{align*}
		L_{\alpha+\hat{e}_i}:=&\ [L_\alpha,A_i],\quad\mbox{where }\substack{\hat{e}_i=\\\ }\substack{(0,\ldots,1,\ldots,0)\\i }\mbox{ for }i=1,\ldots,n.
		\end{align*}
		There exist $C\geq0$ for $i=1,\ldots,n$ and $r\in[0,\infty)$ with $Cr<\omega/(nC_L^2)$ such that 
		\begin{equation*}
		\NN{L_\alpha y}_{X}\leq C\alpha!r^{|\alpha|}\sum_{i=1}^n\NN{A_i y}_{X}
		\end{equation*}
		for all $\alpha\in\mathbb N_0^n$ and all $y\in Y$. 
	\end{enumerate}
	\item[(H3)] \label{hypo3} Define $Q(y):=F(y)-F(\bar x)-DF(\bar x)y$.
	\begin{itemize}
		\item[(H3a)] We assume that
		\begin{equation*}
		\NN{A^\alpha Q(y)}_{X}\leq  \sum_{\gamma\leq \alpha}\binom{\alpha}{\gamma}\NN{A^{\alpha-\gamma}y}_{X} \sum _{\N\beta\leq 1}M_{|\beta|}\NN{A^{\gamma+\beta}y}_{X}
		\end{equation*}
		for all $\alpha\in\mathbb N_0^n$, $y\in Y$ and some $M_{\beta}\geq0$.
		\item[(H3b)] We have the following Lipschitz estimate
				\begin{multline*}
		\NN{A^\alpha (Q(y)- Q(x))}_{X}\leq  \sup_{z,z'\in[x,y]}\sum_{\gamma\leq \alpha}\binom{\alpha}{\gamma}\bigg(\NN{A^{\alpha-\gamma}z}_{X} \sum _{\N\beta\leq 1}M'_{|\beta|}\NN{A^{\gamma+\beta}(y-x)}_{X}\\+\NN{A^{\alpha-\gamma}(y-x)}_{X} \sum _{\N\beta\leq 1}M'_{|\beta|}\NN{A^{\gamma+\beta}z'}_{X}\bigg)
		\end{multline*}
		for all $\alpha\in\mathbb N_0^n$, $y\in Y$ and some $M'_{\beta}\geq0$, where $[x,y]:=\{sx+(1-s)y:s\in[0,1]\}$.
	\end{itemize} 
\end{enumerate}

We now generalize Definition \ref{def2} for these properties.
\begin{definition}
	Let $\nu\in [0,\infty)$. We define
	\begin{equation*}
	\NN{y}_{X^{\nu}_{0}}:= \sum_{\alpha\in\mathbb N_0^n} \frac{\nu^{|\alpha|}}{\alpha!}\NN{A^{\alpha}y}_{X}
	\end{equation*}
	for $y\in Y$. 
\end{definition}

\begin{theorem}\label{thm.of.main.thm1}
	Assume that (H1)-(H3) hold.  Then for every positive $\nu_0<\frac1r(1-\sqrt[n]{\frac{nCC_L^2r}{\omega}})$, there exists an $\varepsilon>0$ such that if 
	\begin{equation}\label{con.on.u_0.8}
	\|x_0-\overline x\|_{X_0^{\nu}}\leq \varepsilon \nu
	\end{equation} 
	for some $\nu\leq \nu_0$, then \eqref{d_tx=F} has a strong solution $x$ with $x|_{t=0}=x_0$ satisfying
	\begin{equation*}
	\|x(t)-\overline x\|_{X_0^{\nu e^{-\omega t}}}\leq 2C_Le^{-\omega t}\varepsilon\nu\quad\mbox{for all }t\geq0.
	\end{equation*}
	Moreover, for all $x_0,y_0\in Y$ fulfilling \eqref{con.on.u_0.8}, we have 
	\begin{equation*}
	\|x(t)-y(t)\|_{X_0^{\nu e^{-\omega t}}}\leq 2C_Le^{-\omega t}\|x_0-y_0\|_{X_0^{\nu}}\quad\mbox{for all }t\geq0,
	\end{equation*}
	where $x,y$ are the solution of \eqref{d_tx=F} with $x(0)=x_0$ and $y(0)=y_0$, respectively.
\end{theorem}
\begin{remark}
	The operator $L_\alpha$ is well-define, because
	\begin{align*}
	[L_{\alpha+\hat e_i},A_j]=[[L_\alpha,A_i],A_j]=-[[A_j,L_\alpha],A_i]-[[A_i,A_j],L_\alpha]=[[L_\alpha,A_j],A_i]=[L_{\alpha+\hat e_j},A_i]
	\end{align*}
	for $i,j=1,\ldots,n$ according to the Jacoby identity and the the assumption $[A_i,A_j]=0$.
\end{remark}

\begin{example}
	Let $\tilde Q:D(A)\times D(A) \to X$ be bilinear fulfilling
	\begin{equation}\label{Norm.Q}
	\NN{\tilde Q(x,y)}_X\leq C_Q\sum_{i=1}^{n}\left(\NN{A_ix}_X\NN{y}_X+\NN{x}_X\NN{A_iy}_X\right)
	\end{equation}
	as well as
	\begin{equation}\label{A.B.Diff.eigenschaft}
	A_i\tilde Q(x,y)=\tilde Q(A_ix,y)+\tilde Q(x,A_iy)\quad\mbox{for all }x,y\in Y,\ i=1,\ldots,N
	\end{equation}
	and some $C_Q$. Then it holds 
	\begin{equation*}
	\NN{A^\alpha\tilde Q(x,x)}_X\leq 2C_Q\sum_{i=0}^n\sum_{\beta\leq \alpha}\binom{\alpha}{\beta}\NN{A^{\alpha-\beta}x}_X\NN{A^{\beta+\hat{e}_i}x}_X
	\end{equation*}
	and
	\begin{multline*}
	\NN{A^\alpha (\tilde Q(x,x)- Q(y,y))}_{X}\leq  2C_Q\sup_{z\in\{x,y\}}\sum_{i=1}^n\sum_{\gamma\leq \alpha}\binom{\alpha}{\gamma}\bigg(\NN{A^{\alpha-\gamma}z}_{X} \NN{A^{\gamma+\hat{e}_i}(y-x)}_{X}\\+\NN{A^{\alpha-\gamma}(y-x)}_{X} \NN{A^{\gamma+\hat{e}_i}z}_{X}\bigg)
	\end{multline*}	
	for all $\alpha,\beta\in\mathbb N_0^n$, $x,y\in Y$. In particular, $Q(y):=\tilde Q(y,y)$ satisfies the assumption of (H3) with $M_0=\tilde M_0=0$, $M_1=\tilde M_1=2C_Q$. 
\end{example}
\begin{proof}
	According to the Leibniz formula, it holds
	\begin{align}\notag
	\NN{A^\alpha\tilde Q(x,y)}_X&\leq
	\sum_{\beta\leq \alpha}\binom{\alpha}{\beta}\NN{Q(A^\beta x,A^{\alpha-\beta}y)}_X
	\\&\leq
	C_Q \sum_{i=0}^n\sum_{\beta\leq a}\binom{\alpha}{\beta}\left(\NN{A^{\beta+\hat{e}_i}x}_X\NN{A^{\alpha-\beta}y}_X+\NN{A^\beta x}_X\NN{A^{\alpha-\beta+\hat{e}_i}y}_X\right).\label{in.proof.1}
	\end{align}
	This implies the first assertion setting $y=x$. Since $\tilde Q$ is bilinear, we have
	\begin{equation*}
	\NN{A^\alpha (\tilde Q(x_1,y_1)- \tilde Q(x_2,y_2))}_{X}\leq 
	\NN{A^\alpha\tilde Q(x_1-x_2,y_1)}_{X}+\NN{\tilde A^\alpha Q(x_2,y_2-y_1)}_{X}.
	\end{equation*}
	This implies directly the second assertion using \eqref{in.proof.1}.
\end{proof}

\section{Preliminary commutator estimates for $L$}
\begin{lemma}\label{lem.L,AB}
	Let $\alpha\in\mathbb N_0^n$. Then	
	\begin{equation*}
	[L,A^\alpha]=\sum_{0\neq\gamma\leq \alpha}\binom{\alpha}{\gamma}(-1)^{\N{\gamma}-1}L_{\gamma}A^{\alpha-\gamma}.
	\end{equation*}
\end{lemma}
\begin{proof}
	The assertion is trivial for $\N{\alpha}\leq1$. Let $i\in\{1,\ldots,n\}$. We compute 
	\begin{align*}
	[L,A^{\alpha+\hat{e}_i}] &= [L,A^{\hat{e}_i}]A^\alpha+A^{\hat{e}_i}[L,A^\alpha]\\
	&= L_{\hat{e}_i} A^\alpha+\sum_{0\neq\gamma\leq\alpha}\binom{\alpha}{\gamma}(-1)^{\N{\gamma}-1}A^{\hat{e}_i}L_{\gamma}A^{\alpha-\gamma}\\
	&= L_{\hat{e}_i} A^\alpha+\sum_{0\neq\gamma\leq\alpha}\binom{\alpha}{\gamma}(-1)^{\N{\gamma}-1}(L_{\gamma}A^{\alpha+\hat{e}_i-\gamma}-L_{\gamma+\hat{e}_i}A^{\alpha-\gamma})\\
	&=\sum_{0\neq\gamma\leq \alpha}\binom{\alpha}{\gamma}(-1)^{\N{\gamma}-1}L_{\gamma}A^{\alpha+\hat{e}_i-\gamma}+\sum_{\hat{e}_i\lneqq\beta\leq \alpha+\hat{e}_i}\binom{\alpha}{\beta-\hat{e}_i}(-1)^{\N{\beta-\hat{e}_i}-1}L_{\beta}A^{\alpha+\hat{e}_i-\beta}\\
	&=\sum_{0\neq\gamma\leq \alpha+\hat{e}_i}\binom{\alpha+\hat{e}_i}{\gamma}(-1)^{\N{\gamma}-1}L_{\gamma}A^{\alpha+\hat{e}_i-\gamma}.\qedhere
	\end{align*}
\end{proof}
\begin{lemma}\label{lem.R_N.leq.nu.Q_N}
	Let $C,r$  be as in (H2b). Then for $\nu<1/r$ it holds
		\begin{equation*}
	\sum_{\alpha\leq N}\frac{\nu^{|\alpha|}}{\alpha!}	\NN{[L,A^\alpha]y}_{X}
	\leq
	\frac{nC\nu r}{(1-\nu r)^n}\sum_{\alpha\lneqq N}\frac{\nu^{|\alpha|}}{\alpha!}\sum_{i=1}^n\NN{A^{\alpha+\hat{e}_i}y}_{X}
	\end{equation*}
	 for $y\in Y$ and $N\in\mathbb N_0^n$. 
\end{lemma}
\begin{proof}
 Let $\NN{\cdot}_X':= C\sum_{i=1}^n\NN{A^{\hat{e}_i}\cdot}_{X}$. Using Lemma \ref{lem.L,AB} and the hypothesis \eqref{hypo3}, we have
	\begin{align*}
	\NN{[L,A^\alpha]x}_{X}&\leq \sum_{0\neq\gamma\leq\alpha}\binom{\alpha}{\gamma}\NN{L_{\gamma}A^{\alpha-\gamma}x}_{X}
	\\&\leq
	\sum_{0\neq\gamma\leq\alpha}\frac{\alpha!r^{|\gamma|}}{(\alpha-\gamma)!}\NN{A^{\alpha-\gamma}x}_X'
	\\&=
	\alpha!\sum_{\gamma\lneqq\alpha}\frac{r^{|\alpha-\gamma|}}{\gamma!}\NN{A^{\gamma}x}_X'.
	\end{align*}
	Define $\delta=\nu r$. Then for $N\in\mathbb N_0^n$ and $i=1,\ldots,n$, it holds
	\begin{align*}
	\sum_{\alpha\leq N}\frac{\nu^{|\alpha|}}{\alpha!}	\NN{[L,A^\alpha]x}_{X}
	&
	\leq 
	\sum_{\alpha\leq N}\sum_{\gamma\lneqq\alpha}\delta^{|\alpha-\gamma|}\frac{\nu^{|\gamma|}}{\gamma!}\NN{A^{\gamma}x}_{X}'
	\\&
	= 
	\delta\sum_{\hat{e_1}\leq\alpha \leq N}\sum_{\gamma\leq \alpha-\hat{e}_1}\delta^{|\alpha-\hat{e}_1-\gamma|}\frac{\nu^{|\gamma|}}{\gamma!}\NN{A^{\gamma}x}_{X}'	+ 
	\sum_{\alpha\leq N}\sum_{\substack{\gamma\lneqq\alpha\\\gamma_1=\alpha_1}}\delta^{|\alpha-\gamma|}\frac{\nu^{|\gamma|}}{\gamma!}\NN{A^{\gamma}x}_{X}'
	\\&
	= 
	\sum_{i\leq n}\delta\sum_{\hat{e_i}\leq\alpha \leq N}\sum_{\substack{\gamma\leq \alpha-\hat{e}_i\\\gamma_k=\alpha_k,\ k< i}}\delta^{|\alpha-\hat{e}_i-\gamma|}\frac{\nu^{|\gamma|}}{\gamma!}\NN{A^{\gamma}x}_{X}'
	\\&
	\leq
n\delta\sum_{\alpha \lneqq N}\delta^{|\alpha|}\sum_{\gamma\lneqq N}\frac{\nu^{|\gamma|}}{\gamma!}\NN{A^{\gamma}x}_{X}'
	\end{align*}
	using the Cauchy-product for finite sums. 	Thus, we obtain the assertion by estimating $\sum_{\alpha \lneqq N}\delta^{|\alpha|}\leq \frac{1}{(1-\delta)^n}$.
\end{proof}
\section{Time depending collisions}

Instead of the norm $\NN{\cdot}_X$ and the r.h.s.~$Q$, we can also use a time depending norm $\NN{\cdot}_{X_t}$ on $Y$ and a time depending collision operator $Q_t$, respectively. Then we need the following assumptions. 

Let $L$ be a generator of a strong continuous group $e^{tL}$ on $X$. There exists $C,r\geq0$ such that
\begin{equation}\tag{H2'}
\NN{e^{tL}L_\alpha y}_{X_t}\leq C \alpha!r^{|\alpha|}\sum_{i=1}^n\NN{e^{tL}A_i y}_{X_t}
\end{equation}
for all $\alpha\in\mathbb N_0^n$ and all $y\in Y$, where $L_0=L$ and $L_{\alpha+\hat{e}_i}:=[L_\alpha,A_{\hat{e}_i}]$.

Moreover, we assume that
\begin{equation}\tag{H3a'}
\NN{e^{tL}A^\alpha Q_t(y)}_{X_t}\leq  e^{-\omega t}\sum_{\gamma\leq \alpha}\binom{\alpha}{\gamma}\NN{e^{tL}A^{\alpha-\gamma}y}_{X_t} \sum _{\N\beta\leq 1}M_{|\beta|}\NN{e^{tL}A^{\gamma+\beta}y}_{X_t}
\end{equation}
holds for all $t>0$,  $\alpha\in\mathbb N_0^n$, $y\in Y$ and some $M_{\beta}\geq0$ and some $\omega>Cr$.
\begin{multline}\tag{H3b'}
\NN{e^{tL}A^\alpha (Q_t(y)- Q_t(x))}_{X}\\\leq  e^{-\omega t}\sup_{z,z'\in[x,y]}\sum_{\gamma\leq \alpha}\binom{\alpha}{\gamma}\bigg(\NN{A^{\alpha-\gamma}z}_{X_t} \sum _{\N\beta\leq 1}M'_{|\beta|}\NN{e^{tL}A^{\gamma+\beta}(y-x)}_{X_t}\\+\NN{A^{\alpha-\gamma}(y-x)}_{X_t} \sum _{\N\beta\leq 1}M'_{|\beta|}\NN{e^{tL}A^{\gamma+\beta}z'}_{X_t}\bigg)
\end{multline}
for all $t>0$, $\alpha\in\mathbb N_0^n$, $y\in Y$ and some $M'_{\beta}\geq0$, where $[x,y]:=\{sx+(1-s)y:s\in[0,1]\}$. Moreover, we need an estimate on the time derivative of the norm, i.e.,
\begin{equation}
\tag{H4'}
\del_t\NN{x(t)}_{X_t}\leq \NN{\del_t x(t)}_{X_t}
\end{equation}
for all $x\in C^1([0,\infty),X)$. 
\begin{lemma}\label{lem7}
	For $\|\cdot\|_{X_t}:=\|\cdot\|_X$ the modified hypothesis (H2')-(H4') are a consequence of the original ones (H2)-(H3) since $\|e^{tL}\|_{L(X)}\leq C_Le^{\omega t}$ for $t\in\mathbb R$. Note that, we have to multiply the constant $C$ from $(H2b)$ by $C_L^2$ to obtain the constant of (H2').
\end{lemma}
With the same arguments as in the proof of Lemma \ref{lem.R_N.leq.nu.Q_N}, 
we can prove its corresponding version:
\begin{lemma}\label{lem.R_N.leq.nu.Q_N2}
	Let $C,r$  be as in (H2'). Then for $\nu<1/r$, it holds
	\begin{equation*}
	\sum_{\alpha\leq N}\frac{\nu^{|\alpha|}}{\alpha!}	\NN{e^{tL}[L,A^\alpha]y}_{X_t}
	\leq
	\frac{nC\nu r}{(1-\nu r)^n}\sum_{\alpha\lneqq N}\frac{\nu^{|\alpha|}}{\alpha!}\sum_{i=1}^n\NN{e^{tL}A^{\alpha+\hat{e}_i}y}_{X_t}
	\end{equation*}
	for $y\in Y$ and $N\in\mathbb N_0^n$. 
\end{lemma}
\section{Transformed equation}
\color{black}
As in the previous section, we may assume that $Q=Q_t$ depends directly on time and that we have a time depending norm such that (H2')-(H4') are fulfilled. 
\begin{definition}[Transformation of the equation]
 For $t\in \R$ and $y\in Y$, we define 
	\[
	A_{tL}:= e^{tL}A e^{-tL}\ \mbox{and }\ Q_{tL}(y):= e^{tL}Q_t(e^{-tL}y).
	\]
	Thus, if $u$ is a solution of
	\begin{equation}\label{d_tu=Q_t}
	\del_t u = Q_{tL}(u)\quad\mbox{with }u(0)= u_0:=x_0-\bar x,
	\end{equation}
	then $x(t):= \bar x+e^{-tL}u(t)$ solves \eqref{d_tx=F} with $x(0)=x_0$.
\end{definition} 
The main strategy in this paper is to solve \eqref{d_tu=Q_t} by using the following time depended  analytic semi-norms, which are a generalization of the norms found in \cite{MoVi11}.
\begin{definition}
	Let $\nu\in [0,\infty)$. We define
	 \begin{equation*}
	\NN{y}_{X^{\nu}_{t}}:= \sum_{\alpha\in\mathbb N_0^n} \frac{\nu^{|\alpha|}}{\alpha!}\NN{A_{tL}^{\alpha}y}_{X}
	\end{equation*}
	for $y\in Y$ and $t\in \R$. 
\end{definition}
\begin{lemma}\label{lem.submul} Let $y\in Y$ and $t\in\R$, $\nu\geq0$. Then
	\begin{equation*}
	\NN{Q_{tL}(y)}_{X^{\nu}_{t}}\leq 
	e^{-\omega t}\NN{y}_{X^{\nu}_{t}}\sum_{\N{\beta}\leq 1}M_{|\beta|}\NN{A_{Lt}^\beta y}_{X^{\nu}_{t}}.
	\end{equation*}
\end{lemma}
\begin{proof}
	We start making use of (H3a') and the multinomial formula to see
	\begin{align*}
	\NN{Q_{tL}(y)}_{X^{\nu}_{t}}
		&
		=
		\sum_{\alpha\in\mathbb N_0^n} \frac{\nu^{\alpha}}{\alpha!}\NN{e^{tL}A^{\alpha}Q(e^{-tL}y)}_{X_t}
	\\&\leq e^{-\omega t}\sum_{\alpha\in\mathbb N_0^n} \frac{\nu^{\alpha}}{\alpha!}\NN{e^{tL}A^{\alpha}e^{-tL}y}_{X_t}
	\sum_{\alpha\in\mathbb N_0^n} \frac{\nu^{\alpha}}{\alpha!}\sum_{\N{\beta}\leq 1}M_{|\beta|}\NN{e^{tL}A^{\alpha+\beta}e^{-tL}y}_{X_t}
	\\&=
		e^{-\omega t}\NN{y}_{X^{\nu}_{t}}\sum_{\N{\beta}\leq1}M_{|\beta|}\NN{A_{tL}^\beta y}_{X^{\nu}_{t}}. \qedhere
	\end{align*}
\end{proof}
Likewise, we can show the following Lipschitz estimate using (H3b') instead of (H3a').
\begin{lemma}\label{lem.submul.Lip} Let $y_1,y_2\in Y$ and $t\in\R$, $\nu\geq0$. Then
	\begin{multline*}
	\NN{Q_{tL}(y_2)-Q_{tL}(y_2)}_{X^{\nu}_{t}}\\\leq 
	e^{-\omega t}\sum_{\N{\beta}\leq 1}M_{|\beta|}\left(\NN{y_2}_{X^{\nu}_{t}}\NN{A_{Lt}^\beta (y_2-y_1)}_{X^{\nu}_{t}}+\NN{y_2-y_1}_{X^{\nu}_{t}}\NN{A_{Lt}^\beta y_1}_{X^{\nu}_{t}}\right).
	\end{multline*}
\end{lemma}
\begin{proposition}\label{prop.time.estimate}
	Let $\nu_0<1/r$ and
	\[
	\mu\geq\mu_0:=\frac{nCr}{(1-\nu_0r)^n},
	\]
	where $C$ is given by (H2'). 
	We define $\nu(t)=\nu_0\exp(-\mu t)$. Then
	 \begin{multline*}
	 \NN{u(t)}_{\dot X^{\nu(t)}_{Lt}}+\sum_{i}^n\int_s^t\left(\mu-\mu_0\right)\nu(\tau) \NN{A_{tL}^{\hat{e}_i}u(\tau)}_{X^{\nu(\tau)}_{L\tau}}d\tau
	 \\
	 \leq \NN{u(s)}_{\dot X^{\nu(s)}_{Ls}}+\int_s^t\NN{\del_tu(\tau)}_{\dot X^{\nu(\tau)}_{L\tau}}d\tau
	 \end{multline*}
	 for $t>s\geq0$ and $u\in C^0([0,\infty),X)$ such that $u(t)\in Y$ for all $t\geq0$ and $t\mapsto A_{tL}^\alpha u(t)\in C^1((0,\infty),X)$ for all $\alpha\in\mathbb N_0^n$.
\end{proposition}
\begin{proof}
		Let $0\leq s<t<\infty$. At first, we may assume that
	\begin{equation*}
	\tau\mapsto \NN{\del_tu(\tau)}_{\dot X^{\nu(\tau)}_{\tau}}\in L^1(s,t),
	\end{equation*}
	because the assertion is trivial otherwise.
	 For $\lambda\in[0,\infty)$, we define
	\[P_{u,N}(\lambda,t):=\sum_{0\lneqq\alpha\leq N}\frac{\lambda^{|\alpha|}}{\alpha!}\NN{A^{\alpha}_{tL}u(t)}_{X_t}\]
	and
	\[Q_N(\lambda,t):=\sum_{i=1}^n\sum_{\alpha\lneqq N}\frac{\lambda^{|\alpha|}}{\alpha!}\NN{A^{\alpha+\hat{e}_i}_{tL}u(t)}_{X_t}\]
	as well as
	\[R_N(\lambda,t):=\sum_{0\lneqq \alpha\leq N}\frac{\lambda^{|\alpha|}}{\alpha!}\NN{e^{tL}[L,A^{\alpha}]e^{-tL}u(t)}_{X_t}.\]
	Thus, $P_{u,N}(\lambda,t)\to \NN{u(t)}_{\dot X^\lambda_{t}}$ and $Q_N(\lambda,t)\to\sum_{i=1}^n  \NN{A_{tL}^{\hat{e}_i}u(t)}_{X^\lambda_{t}}$ as $N\to \infty$. 
	Let $\alpha\in\mathbb N_0^n$ and $0<s<t$. Since $\partial_t\|\cdot\|_{X_t}\leq \|\partial_t\cdot\|_{X_t}$, we have
	\begin{align*}
	&\N{\NN{A^{\alpha}_{tL}u(t)}_{X_t}-\NN{A^{\alpha}_{sL}u(s)}_{X_s}}
	\newln\leq \sup_{s\leq \tau\leq t}\NN{
		\del_\tau A^{\alpha}_{\tau L}u(\tau))
	}_{X_\tau}(t-s)
\newln\leq \sup_{s\leq \tau\leq t}\NN{
A^{\alpha}_{\tau L}\del_\tau u(\tau)
}_{X_\tau}(t-s)+ \sup_{s\leq \tau\leq t}\NN{
	e^{\tau L}[L,A^{\alpha}]e^{-\tau L}u(\tau))
}_{X_\tau}(t-s)
 	\end{align*}
	using
	\begin{align*}
	[\del_t ,A^{\alpha}_{tL}]y
	 &=e^{Lt}A^{\alpha}[\del_t,e^{-tL}]y+[\del_t,e^{L t}]A^{\alpha}e^{-tL}y
	 \\&=-e^{L t}LA^{\alpha}e^{-tL}y+e^{L t}A^{\alpha}Le^{-tL}y
	 \\&=e^{L t}[L,A^{\alpha}]e^{-tL}y
	\end{align*}
	for $y\in Y$. This implies
	\[\N{P_{u,N}(\lambda,t)-P_{u,N}(\lambda,s)}\leq  \sup_{s\leq \tau\leq t}\left(P_{\del_t u,N}(\lambda,\tau)+R_N(\lambda,\tau)\right)(t-s).\]
 The  estimate
\begin{equation*}
\frac{\nu(t)^{|\alpha|}-\nu(s)^{|\alpha|}}{{\alpha}!}\geq \sum_{i=1}^n\frac{\nu(t)^{|\alpha-\hat e_i|}}{(\alpha-\hat e_i)!}(\nu(t)-\nu(s))
\end{equation*}
entails
\begin{align*}
\N{P_{u,N}(\nu(t),t)-P_{u,N}(\nu(s),s)}&\leq  \sup_{s\leq \tau\leq t}\left(P_{\del_t u,N}(\nu(s),\tau)+R_N(\nu(s),\tau)\right) (t-s)\newln+ \sup_{s\leq \tau\leq t}{\dot\nu(\tau)  Q_N(\nu(t),t)}(t-s).
\end{align*}
Thus, $P_{u,N}(\nu(t),t)$ is Lipschitz continuous w.r.t.~$t$ and belongs to $W^{1,\infty}((0,T))$ with
\[\frac{d}{dt}P_{u,N}(\nu(t),t)\leq P_{\del_t u,N}(\nu(t),t)+R_N(\nu(t),t)+\dot\nu(t) Q_N(\nu(t),t),\]
since $P_{u,N}$, $P_{\del_tu,N}$ and $Q_N$ are continuous. By $P_{\del_t u,N}(\nu(\tau),\tau)\leq\NN{\del_tu(\tau)}_{X^{\nu(\tau)}_{L\tau}}\in L^1(0,T)$, the dominated convergence theorem implies  $\int_0^TP_{\del_t u,N}(\nu(\tau),\tau)d\tau\to \int_0^T\NN{\del_tu(\tau)}_{X^{\nu(\tau)}_{L\tau}}d\tau$ as $N\to\infty$. According to the monotone convergence theorem we have \[
\int_s^t\nu(\tau) Q_{N}(\nu(\tau),\tau)d\tau \to\sum_{i=1}^{n} \int_s^t\nu(\tau)\NN{A_{\tau L}^{\hat{e}_i}u(\tau)}_{X^{\nu(\tau)}_{L\tau}}d\tau.\]
Then Lemma \ref{lem.R_N.leq.nu.Q_N2} yields 
\begin{align*}
R_N(\nu(t),t)
&\leq 
\frac{nC\nu(t) r}{(1-\nu(t) r)^n}\sum_{\alpha\lneqq N}\frac{\nu(t)^{|\alpha|}}{\alpha!}\sum_{i=1}^n\NN{e^{tL}A^{\alpha+\hat{e}_i}e^{-tL}u(t)}_{X}
\\&\leq \mu_0\nu(t)  Q_N(\nu(t),t)
\end{align*}
for $N\in\mathbb N_0^n$ recalling  $\mu_0:=\frac{nC}{(1-\nu_0r)^n}$ and $\nu(t)\leq \nu_0$.  Finally, we obtain 
	\begin{align*}
&\NN{u(t)}_{\dot X^{\nu(t)}_{t}}+\sum_{i=1}^n\int_s^t\left((\mu-\mu_0)\nu(\tau) \NN{A_{tL}^{\hat{e}_i}u(\tau)}_{X^{\nu(\tau)}_{\tau}}-\NN{\del_tu(\tau)}_{\dot X^{\nu(\tau)}_{\tau}}\right)d\tau
\newln \leq \limsup_{N\to\infty}
P_{u,N}(\nu(t),t)-\int_s^t\left(\dot\nu(\tau) Q_N(\nu(\tau),\tau)+R_N(\nu(\tau),\tau)+P_{\del_t u,N}(\nu(\tau),\tau)\right) d\tau
\newln\ \ 
\leq \NN{u(s)}_{\dot X^{\nu(s)}_{s}}.
\end{align*}
This finishes the proof using $\dot\nu=-\mu\nu$.
\end{proof}
\begin{definition}
	Let $C^i_L([0,\infty);Y):= \{u:[0,\infty)\to Y \mbox{ s.t.~}t\mapsto A_{tL}^\alpha u(t)\in C^i([0,\infty),X),\ \alpha\in\mathbb N_0^n\}$ for $i\in\mathbb N_0$. We define $\Phi:C^0_L([0,\infty);Y)\to C^0_L([0,\infty);Y)$  by
	\[
	\Phi(u)(t):=u(0)+\int_0^t Q_{\tau L}(u(\tau))d\tau.
	\]
Let $\nu_0<1/r$ and
\[
\mu\geq\mu_0:=\frac{nCr}{(1-\nu_0r)^n}
\]
with $C$ as in (H2'). We define $\nu(t)=\nu_0\exp(-\mu t)$ and
\begin{align*}
\NN{u}_{\nu_0,\mu}:=\sup_{t\geq0}\left(\NN{u(t)}_{X^{\nu(t)}_t}+(\mu-\mu_0)\sum_{i=1}^n\int_0^t\nu(\tau)\NN{A^{\hat{e}_i}_{L\tau}u(\tau)}_{X_\tau^{\nu(\tau)}}d\tau\right)
\end{align*}
for $u\in C^0_L([0,\infty);Y)$.
\end{definition}
\begin{lemma}\label{lem.norm.estimate}	Let $\nu_0<1/r$ and assume that $\omega>\mu_0$. Then
	\[
	\NN{\Phi(u)}_{\nu_0,\omega}\leq
	\NN{u(0)}_{ X^{\nu_0}_{0}}+\max\left\{\frac{M_0}{\omega},\frac{M_{1}}{\nu_0(\omega-\mu_0)} \right\}\NN{u}_{\nu_0,\omega}^2.
	\]
\end{lemma}
	\begin{proof} Applying Proposition \ref{prop.time.estimate}
		to $\Phi(u)$, we obtain
 \begin{align*}
\NN{\Phi(u)}_{\nu_0,\omega}&= \sup_{t\geq0}\left(\NN{\Phi(u)(t)}_{ X^{\nu(t)}_{t}}+(\omega-\mu_0)\sum_{i=1}^n\int_0^t\nu(\tau) \NN{A_{tL}^{\hat{e}_i}\Phi(u)(\tau)}_{X^{\nu(\tau)}_{\tau}}d\tau\right)
\\
&\leq \NN{\Phi(u)(0)}_{ X^{\nu(0)}_{0}}+\int_0^\infty\NN{\partial_t \Phi(u(\tau))}_{ X^{\nu(\tau)}_{\tau}}d\tau\\
\\
&= \NN{u(0)}_{ X^{\nu(0)}_{0}}+\int_0^\infty\NN{Q_{\tau L}(u(\tau))}_{ X^{\nu(\tau)}_{\tau}}d\tau.
\end{align*}
Thus,
\begin{align*}
\NN{\Phi(u)}_{\nu_0,\omega}
&\leq \NN{u(0)}_{ X^{\nu(0)}_{0}}+
\int_0^\infty
e^{-\omega \tau}\NN{u(\tau)}_{X^{\nu}_{\tau}}\sum_{\N{\beta}\leq 1}M_{|\beta|}\NN{A_{L\tau}^\beta u(\tau)}_{X^{\nu}_{\tau}}d\tau\\
&\leq \NN{u(0)}_{ X^{\nu(0)}_{0}}+\NN{u}_{\nu_0,\omega}\sum_{\N{\beta}\leq 1}
\int_0^\infty
e^{-\omega \tau}M_{|\beta|}\NN{A_{L\tau}^\beta u(\tau)}_{X^{\nu}_{\tau}}d\tau.
\end{align*}
For $\beta=0$, we estimate
\begin{align*}
\int_0^\infty
e^{-\omega \tau}M_{0}\NN{A_{L\tau}^\beta u(\tau)}_{X^{\nu}_{\tau}}d\tau
&=
\int_0^\infty
e^{-\omega \tau}d\tau \sup_{0\leq\tau<\infty}M_{0}\NN{ u(\tau)}_{X^{\nu}_{\tau}}
\leq\frac{M_0}{\omega}\| u\|_{\nu_0,\omega}.
\end{align*}
In the remaining cases where $|\beta|=1$, we have
\begin{align*}
\int_0^\infty
e^{-\omega \tau}M_{1}\NN{A_{L\tau}^\beta u(\tau)}_{X^{\nu}_{\tau}}d\tau
&\leq 
\frac{M_1}{\nu_0(\omega-\mu_0)}(\omega-\mu_0)\int_0^\infty
\nu_0e^{-\omega \tau}\NN{A_{L\tau}^\beta u(\tau)}_{X^{\nu}_{\tau}}d\tau
\\&\leq \frac{M_1}{\nu_0(\omega-\mu_0)}\| u\|_{\nu_0,\omega}.
\end{align*}
Finally, we conclude with
\begin{align*}
\NN{\Phi(u)}_{\nu_0,\omega}
&\leq 
  \NN{u(0)}_{ X^{\nu(0)}_{0}}+\max\left\{\frac{M_0}{\omega},\frac{M_{1}}{\nu_0(\omega-\mu_0)} \right\}\NN{u}_{\nu_0,\omega}^2
\end{align*}
the assertion.
\end{proof}
\begin{lemma}With the same hypothesis as in the previous lemma, let $u,v\in C_L^0([0,\infty),Y)$ such that $R=\max\{\NN{u}_{\nu_0,\omega},\NN{v}_{\nu_0,\omega}\}$. Then
			\begin{equation*}
	\NN{\Phi(u)-\Phi(v)}_{\nu_0,\omega}\leq\NN{u(0)-v(0)}_{ X^{\nu(0)}_{0}}
	+ 
		2R\max\left\{\frac{M'_0}{\omega},\frac{M'_1}{\nu_0(\omega-\mu_0)} \right\}\NN{u-v}_{\nu_0,\omega}
	\end{equation*}
\end{lemma}
\begin{proof}
	Let $u,v\in C_L^0([0,\infty);Y)$ such that $\NN{u}_{\nu,\mu},\NN{v}_{\nu,\mu}\leq R$. We have
	\begin{align*}
	\NN{\Phi(u)-\Phi(v)}_{\nu_0,\omega}
		&=
			 \sup_{t\geq0}\bigg(\NN{\Phi(u)(t)-\Phi(v)(t)}_{ X^{\nu(t)}_{t}}
		\\& \quad
	+(\omega-\mu_0)\sum_{i}^n\int_0^t\nu(\tau) \NN{A_{tL}^{\hat{e}_i}(\Phi(u)(\tau)-\Phi(v)(\tau))}_{X^{\nu(\tau)}_{\tau}}d\tau\bigg)
	\\
	&\leq 
		\NN{u(0)-v(0)}_{ X^{\nu(0)}_{0}}+\int_0^\infty\NN{Q_{\tau L}(u(\tau))-Q_{\tau L}(v(\tau))}_{ X^{\nu(\tau)}_{\tau}}d\tau
		\end{align*}
		For the next step, we have to use the condition (H3b) and proceed similarly as in the proof of Lemma \ref{lem.norm.estimate}. Note that for $\xi\in	[v,u]:=\{t\mapsto s(t)v(t)+(1-s(t))u(t):s(t)\in[0,1]\}$, we observe that $\NN{\xi}_{\nu_0,\omega}\leq R$. Thus,
		\begin{multline*}
		\NN{\Phi(u)-\Phi(v)}_{\nu_0,\omega}-\NN{u(0)-v(0)}_{ X^{\nu(0)}_{0}}
	\\
	\begin{aligned}
	&\leq \sup_{\xi\in[u,v]}
		\int_0^\infty	e^{-\omega \tau}\NN{u(\tau)-v(\tau)}_{X^{\nu}_{\tau}}\sum_{\N{\beta}\leq 1}M'_{|\beta|}\NN{A_{L\tau}^\beta \xi(\tau)}_{X^{\nu}_{\tau}}d\tau
	\newln+
		\sup_{\xi\in[u,v]}
		\int_0^\infty	e^{-\omega \tau}\NN{\xi(\tau)}_{X^{\nu}_{\tau}}\sum_{\N{\beta}\leq 1}M'_{|\beta|}\NN{A_{L\tau}^\beta (u(\tau)-v(\tau))}_{X^{\nu}_{\tau}}d\tau
	\\
	&\leq 
	2R\max\left\{\frac{M'_0}{\omega},\frac{M'_1}{\nu_0(\omega-\mu_0)} \right\}\NN{u-v}_{\nu_0,\omega}.
\end{aligned}
	\end{multline*}
This terminates the proof.
\end{proof}
\begin{definition}
	Let $Z$ denote the subspace of $C_L^0([0,\infty),Y)$ such that $\NN{u}_{\nu,\mu}<\infty$ for all $u\in Z$. Note that $Z$ endowed with $\NN{\cdot}_{\nu_0,\omega}$ is a Banach space. For $R>0$, we define $Z_R:=\{u\in Z: \NN{u}_{\nu_0,\omega}\leq R$ and $u(0):=u_0\}$.
\end{definition}

\begin{proposition}\label{prop.general}
	Let $\nu_0<\frac1r(1-\sqrt[n]{\frac{nCr}{\omega}})$, and let $R>0$ satisfy
	\[ C_0:=1-R\max\left\{\frac{M_0}{\omega},\frac{M_{1}}{\nu_0(\omega-\mu_0)} \right\}>0,\quad C_1:=1-2R\max\left\{\frac{M'_0}{\omega},\frac{M'_{1}}{\nu_0(\omega-\mu_0)} \right\}>0,\]
	where $\mu_0=\frac{nCr}{(1-\nu_0r)^n}>\omega$.	Then for all $u_0\in Y$ with  
\begin{equation}\label{con.on.u_0}
	\NN{u_0}_{ X^{\nu_0}_{0}}\leq C_0R,
\end{equation}
	the equation \eqref{d_tu=Q_t} has a unique solution $u$  in $Z_R$ satisfying $u|_{t=0}=u_0$. Moreover, let $u_0,w_0$ satisfy \eqref{con.on.u_0} and let $u,w$ be the solution of \eqref{d_tu=Q_t} with $u(0)=u_0$ and  $w(0)=w_0$, respectively. Then
	\begin{equation*}
	C_1\|u-w\|_{\nu_0,\omega}\leq \|u_0-w_0\|_{X_0^{\nu_0}}.
	\end{equation*}
	
\end{proposition}
\begin{proof}
	We combine the last two lemmata with the Banach fixed-point theorem to see that $\Phi:Z_R\to Z_R$ is a contraction and admits a unique fixed point $u$. By the definition of $Z$ we easily see that $u$ is differentiable with w.r.t.~$t$ in $X$ such that $u$ is a strong solution of \eqref{d_tu=Q_t}.
\end{proof}
\begin{theorem}\label{main.thm1}
	Let $\omega,C,r$ be as in (H2'),(H3a') and (H3b').  Then for every positive $\nu_0<\frac1r(1-\sqrt[n]{\frac{nCr}{\omega}})$, there exists an $\varepsilon>0$ such that if 
	\begin{equation}\label{con.on.u_0.2}
	\|u_0\|_{X_0^{\nu}}\leq \varepsilon \nu
	\end{equation} 
	for some $\nu\leq \nu_0$, then \eqref{d_tu=Q_t} has a strong solution $u$ with $u|_{t=0}=u_0$, with
	\begin{equation*}
	\|u(t)\|_{X_t^{\nu e^{-\omega t}}}\leq 2\varepsilon\nu\quad\mbox{for all }t\geq0.
	\end{equation*}
	Moreover, for all $u_0,w_0\in Y$ satisfying \eqref{con.on.u_0.2}, we have 
		\begin{equation*}
	\|u(t)-w(t)\|_{X_t^{\nu e^{-\omega t}}}\leq 2\|u_0-w_0\|_{X_0^{\nu}}\quad\mbox{for all }t\geq0,
	\end{equation*}
	where $u,w$ are the solution of \eqref{d_tu=Q_t} with $u(0)=u_0$ and $w(0)=w_0$, respectively.
\end{theorem}
\begin{proof}
 First, we recall $\mu_0:=\frac{nCr}{(1-\nu_0r)^n}$. We choose $R'>0$ such that
		\[ R'<\min\left\{\frac{\omega}{M_0\nu_0},\frac{\omega-\mu_0}{M_{1}} \right\},\quad R'\leq\frac14\min\left\{\frac{\omega}{M'_0\nu_0},\frac{\omega-\mu_0}{M'_{1}} \right\}.\]
		With this, we define $R:=R'\nu$ and $\varepsilon:=1-R'\max\left\{\frac{M_0\nu_0}{\omega},\frac{M_{1}}{\omega-\mu_0} \right\}$. Thus,
		\[C_0:=1-R\max\left\{\frac{M_0}{\omega},\frac{M_{1}}{\nu(\omega-\mu_0)} \right\}\geq \varepsilon>0.\]
		Likewise,
		\[C_1:=1-2R\max\left\{\frac{M'_0}{\omega},\frac{M'_{1}}{\nu(\omega-\mu_0)} \right\}\geq1-2R'\max\left\{\frac{M'_0\nu_0}{\omega},\frac{M'_{1}}{(\omega-\mu_0)} \right\}\geq\frac12>0.\]
		Finally, we obtain the assertion by applying the theorem. Note that the estimate \begin{equation*}
		\|u(t)\|_{X_t^{\nu e^{-\omega t}}}\leq 2\varepsilon\nu\quad\mbox{for all }t\geq0
		\end{equation*}
		is a direct consequence of
		\begin{equation*}
		\|u(t)-w(t)\|_{X_t^{\nu e^{-\omega t}}}\leq 2\|u_0-w_0\|_{X_0^{\nu}}\quad\mbox{for all }t\geq0
		\end{equation*}
		for $w(t)=w_0=0$ and $\|u_0\|_{X_0^{\nu}}\leq \varepsilon \nu$.
\end{proof}
\begin{remark}
	By shrinking $R'>0$ such that 	\[ R'<\min\left\{\frac{\omega}{M_0\nu_0},\frac{\omega-\mu_0}{M_{1}} \right\},\quad R'\leq\frac\gamma{2}\min\left\{\frac{\omega}{M'_0\nu_0},\frac{\omega-\mu_0}{M'_{1}} \right\}\]
	is satisfied for a fixed positive $\gamma\in(0,1)$. We can show similarly as in the previous proof that 
			\begin{equation*}
			\|u(t)-w(t)\|_{X_t^{\nu e^{-\omega t}}}\leq \frac1{1-\gamma}\|u_0-w_0\|_{X_0^{\nu}}\quad\mbox{for all }t\geq0
			\end{equation*}
	if $\varepsilon:=1-R'\max\left\{\frac{M_0\nu_0}{\omega},\frac{M_{1}}{\omega-\mu_0} \right\}$.
\end{remark}
\begin{proof}[Proof of Theorem \ref{thm.of.main.thm1}]
	According to Lemma \ref{lem7}, Theorem \ref{thm.of.main.thm1} is a direct consequence of Theorem \ref{main.thm1} for $u_0:=x_0-\bar x$ and $x(t):=\bar x+e^{-tL}u(t)$: 
\begin{align*}
\|x(t)-\bar x\|_{X_0^{\nu e^{-\omega t}}}&= \sum_{\alpha\in\mathbb N_0^n} \frac{(\nu e^{-\omega t})^{|\alpha|}}{\alpha!}\NN{A^{\alpha}(x(t)-\bar x)}_{X}
\\&\leq C_Le^{-\omega t}
\sum_{\alpha\in\mathbb N_0^n} \frac{(\nu e^{-\omega t})^{|\alpha|}}{\alpha!}\NN{e^{tL}A^{\alpha}e^{-tL}e^{tL}(x(t)-\bar x)}_{X}\\&=C_Le^{-\omega t}\|e^{tL}(x(t)-\bar x)\|_{X_t^{\nu e^{-\omega t}}}
=C_Le^{-\omega t}\|u(t)\|_{X_t^{\nu e^{-\omega t}}}\leq 2\varepsilon\nu C_Le^{-\omega t}
\end{align*}
for all $t\geq0$.  Likewise, we have
\begin{equation*}
\|x(t)-y(t)\|_{X_0^{\nu e^{-\omega t}}}\leq 2C_Le^{-\omega t}\|x_0-y_0\|_{X_0^{\nu}}
\end{equation*}
for every $t\geq0$.
\end{proof}
\section{The model case}
In this section, we consider the model equation \eqref{2.be} with $\lambda=(\lambda_0,\lambda_1)\in\mathbb R^2$, $\lambda_1\geq0$.  The substitution
\begin{equation*}
g(x,p,t):=e^{\frac{t}{\tau}}(f(x,p,t)-\mathcal F_\lambda(p))
\end{equation*}
leads to the system
\begin{equation}\label{be.for.g1}
\left\{\begin{aligned}
\partial_t g+\nabla\eps(p)\cdot\nabla_x g-U\nabla_x\int_{\mathbb T^d}gdp'\cdot\nabla_p\mathcal F_\lambda(p)&=Ue^{-\frac{t}{\tau}}\nabla_x\int_{\mathbb T^d}gdp'\cdot\nabla_pg,
\\ g|_{t=0}&=g_0
\end{aligned}\right.
\end{equation}
for $g_0:=f_0-\mathcal F_\lambda$. Defining
\begin{equation*}
Lf(x,p):=\nabla\eps(p)\cdot\nabla_xf(x,p)-U\nabla_x\int_{\mathbb T^d}f(x,p')dp'\cdot\nabla_p\mathcal F_\lambda(p)
\end{equation*}
and
\begin{equation*}
Q_t(f)(x,p):=Ue^{-\frac{t}{\tau}}\nabla_x\int_{\mathbb T^d}f(x,p')dp'\cdot\nabla_pf(x,p),
\end{equation*}
we can rewrite \eqref{be.for.g1} to
\begin{equation}\label{be.for.g2}
\partial_t g+Lg=Q_t(g).
\end{equation}
The idea is now to apply the general result, which requires the hypothesis (H2')-(H4').
\begin{definition}
	Let $\mathcal S(\mathbb R^d\times\mathbb T^d)$ and $C^\infty_b(\mathbb R^d\times\mathbb T^d)$ be the  Schwartz space and the space of bounded smooth functions, respectively. Let $\lambda=(\lambda_0,\lambda_1)\in\mathbb R^2_+:=\{(x,y)\in\mathbb R^2:y\geq0\}$ and $\mathcal F_\lambda(p)=1/(\eta+e^{-\lambda_0-\lambda_1\eps(p)})$ be the Fermi-Dirac distribution function.
	\begin{itemize}
		\item For $\lambda_1=0$, we can define $Y:= C^\infty_b(\mathbb R^d\times\mathbb T^d)$ and $X:= C^0_b(\mathbb R^d\times\mathbb T^d)$.
		\item For general $\lambda\in \mathbb R^2_+$, let $k\in\mathbb N$,  $k>\frac{d}{2}$. We define $Y:=\mathcal S(\mathbb R^d\times\mathbb T^d)$ and $X:= H^k_xL^2_p(\mathbb R^d_x\times\mathbb T^d_p)$ equipped with the scalar product
		\begin{equation*}
		\langle f,g\rangle_X:=\sum_{|\alpha|\leq k}\langle \partial_x^\alpha f,\partial_x^\alpha g\rangle_0,
		\end{equation*}
		where 
		\begin{equation*}
		\langle f,g \rangle_0:=
		\int_{\mathbb R^d}\int_{\mathbb T^d} f(x,p) g(x,p)\frac{dpdx}{\mathcal F_\lambda(p)(1-\eta \mathcal F_\lambda(p))}+U\lambda_1\int_{\mathbb R^d}\rho_f\rho_g dx
		\end{equation*}
		and $\rho_f(x):=\int_{\mathbb T^d}f(x,p)dp$.
	\end{itemize} 
\end{definition}
\begin{lemma}
\label{X.Algebra}
There exists a $C_\lambda>0$ such that
\[\|\rho_hg\|_{X}\leq C_\lambda \|h\|_X\|g\|_X\]
for $\rho_h(x):=\int_{\mathbb T^d}h(x,p)dp$. For $X=C^0_b(\mathbb R^d\times\mathbb T^d)$, we can easily see that $C_\lambda=1$ using $|\mathbb T^d|=1$. In the other case, the assertion is a consequence of the algebra properties of $H^k$ for $k\geq\frac d2$.
\end{lemma}
\begin{lemma}\label{lem.L.generator}
	$L:D(L)\subset X\to X$ is a generator of a $C_0$ contraction group $(e^{tL})_{t\in\mathbb R}$ with $L(Y)\subset Y\subset D(L)$.
\end{lemma}
The following proof is similar as the proof of Theorem 3.1 of \cite{BaBe13}.
\begin{proof}
	The assertion is clear for $\lambda_1=0$ and $X=C^0_b(\mathbb R^d\times\mathbb T^d)$ since then $(e^{tL})$ is a transport contraction group generated by $L=\nabla\eps(p)\cdot\nabla_x$. Now, let $\lambda\in\mathbb R^2_+$ and $X=H^k_xL^2_p(\mathbb R^d_x\times\mathbb T^d_p)$.
	
	 For $h\in Y:=\mathcal S(\mathbb R^d\times\mathbb T^d)$ with $\|h\|_{X}<\infty$ we have
	 \begin{align*}
	 \langle Lh,h\rangle_X&= \sum_{|\alpha|\leq k}\langle \partial_x^\alpha Lh,\partial_x^\alpha h\rangle_0
	 \\&=\sum_{|\alpha|\leq k}\langle L\partial_x^\alpha h,\partial_x^\alpha h\rangle_0
	 \end{align*}
	 since we can easily show that $[L,\partial_{x_i}]=0$ for $i=1,\ldots,d$. Then abbreviating $g:=\partial_x^\alpha h$, we have
	 \begin{align*}
	 \langle Lg,g\rangle_{0}=&\int_{\mathbb R^d}\int_{\mathbb T^d}\left(\nabla\eps(p)\cdot\nabla_xg(x,p)-U\nabla_x\rho_g(x)\cdot\nabla_p\mathcal F_\lambda(p)\right)g(x,p)\frac{ dpdx}{\mathcal F_\lambda(1-\eta F_\lambda)}
	 \\&+ U\lambda_1\int_{\mathbb R^d}\int_{\mathbb T^d}\left(\nabla\eps(p)\cdot\nabla_xg(x,p)-U\nabla_x\rho_g(x)\cdot\nabla_p\mathcal F_\lambda(p)\right)dp \rho_g(x)dx
	 \\=&\int_{\mathbb R^d}\int_{\mathbb T^d}\frac{\nabla\eps(p)\cdot\nabla_xg(x,p)^2}{2{\mathcal F_\lambda(1-\eta F_\lambda)}}dxdp-\lambda_1U\int_{\mathbb R^d}\nabla_x\rho_g(x)\cdot \int_{\mathbb T^d}\nabla_p\eps(p)g(x,p)dpdx
	 \\&- U\lambda_1\int_{\mathbb R^d}\int_{\mathbb T^d}\nabla\eps(p)\cdot\nabla_xg(x,p)dp\rho_g(x)dx\newln-U\int_{\mathbb R^d}\int_{\mathbb T^d}\mathbb\nabla_x\rho_g(x)\cdot\nabla_p\mathcal F_\lambda(p)dp \rho_g(x)dx=:I_1+I_2+I_3+I_4
	 \end{align*}
	 By the Gau\ss{} law, we see that $I_1=I_4=0$. Moreover, $I_2=-I_3$ implying that $\langle Lg,g\rangle_0=0$ and hence $\langle Lh,h\rangle_X=0$. Thus, $L$ is the closure of an anti-symmetric operator such that
	 \begin{equation*}
	 \|(\sigma+L)h\|_X\|h\|_X\geq |\langle (\sigma+L)h,h\rangle_X|=|\sigma|\|h\|_X
	 \end{equation*}
	 for $\sigma\in \mathbb C$ with $\Re \sigma\neq0$. Next, as in \cite{BaBe13}, we want to show that $L$ is indeed anti-adjoint. For this, we need show for $\sigma\in\mathbb R\setminus\{0\}$ that $(\sigma+L)$ is surjective onto $X$ according to (cf. Theorem V-3.16 or Problem V-3.31 in \cite{Kat66}). Let $h\in Y$. We have to find a solution to the equation
	 \begin{align}\label{eq.sigma+L}
	 \sigma f + Lf=h.
	 \end{align}
	 Applying the Fourier transform w.r.t.~$x$ to \eqref{eq.sigma+L}, we obtain 
	\begin{align*}
	\sigma \hat f(\xi,p)+\nabla_p\eps(p)\cdot i\xi\hat{f}(\xi,p)-Ui\xi\hat \rho_f(\xi)\cdot\nabla_p\mathcal F_\lambda(p)=\hat h(\xi,p),
	\end{align*}
	where $\hat{\rho}_f:=\int_{\mathbb T^d}\hat f dp$ implying
	\begin{align*}
	\hat f=\frac{1}{\sigma+\nabla_p\eps(p)\cdot i\xi}\left(\hat h+iU\xi\hat{\rho}_f\cdot\nabla_p\mathcal F_\lambda\right).
	\end{align*}
	An integration of this equality leads to $\hat \rho_f=\hat{\rho}$ with
	\begin{equation}\label{eq.def.hat.rho}
	\left(1-U\int_{\mathbb T^d}\frac{i\xi\cdot\nabla_p\mathcal F_\lambda(p)}{\sigma+\nabla_p\eps(p)\cdot i\xi}dp\right)\hat\rho(\xi)=\int_{\mathbb T^d}\frac{\hat h(\xi,p)}{\sigma+\nabla_p\eps(p)\cdot i\xi}dp.
	\end{equation}
	Since $\eps(-p)=\eps(p)$ implies $\mathcal F_\lambda(-p)=\mathcal F_\lambda(p)$ and $\nabla\eps(-p)=-\nabla\eps(p)$, we have
	\begin{align*}
	\int_{\mathbb T^d}\frac{i\xi\cdot\nabla_p\mathcal F_\lambda(p)}{\sigma+\nabla_p\eps(p)\cdot i\xi}dp
		&=\lambda_1\int_{\mathbb T^d}\frac{i\xi\cdot\nabla_p\eps(p)}{\sigma+\nabla_p\eps(p)\cdot i\xi}\mathcal F_\lambda(p) (1-\eta\mathcal F_\lambda(p))dp
		\\
		&=\lambda_1\int_{\mathbb T^d}\frac{\sigma i\xi\cdot\nabla_p\eps(p)}{\sigma^2+|\nabla_p\eps(p)\cdot \xi|^2}\mathcal F_\lambda(p) (1-\eta\mathcal F_\lambda(p))dp
		\newln+
		\lambda_1\int_{\mathbb T^d}\frac{|\xi\cdot\nabla_p\eps(p)|^2}{\sigma^2+|\nabla_p\eps(p)\cdot \xi|^2}\mathcal F_\lambda(p) (1-\eta\mathcal F_\lambda(p))dp
		\\&=
		\lambda_1\int_{\mathbb T^d}\frac{|\xi\cdot\nabla_p\eps(p)|^2}{\sigma^2+|\nabla_p\eps(p)\cdot \xi|^2}\mathcal F_\lambda(p) (1-\eta\mathcal F_\lambda(p))dp\geq0.
	\end{align*}
	Thus, we can define $\hat{\rho}$ by \eqref{eq.def.hat.rho} and obtain 
	\begin{align*}
	|\hat \rho(\xi)|\leq \left|\int_{\mathbb T^d}\frac{\hat h(\xi,p)}{\sigma+\nabla_p\eps(p)\cdot i\xi}dp\right|.
	\end{align*}
	We set $\hat f = \frac{1}{\sigma+\nabla_p\eps(p)\cdot i\xi}(\hat h+iU\xi\hat{\rho}_f\cdot\nabla_p\mathcal F_\lambda)$ and have $\hat \rho_f=\hat \rho$. Therefore, we can easily see that there exists a constant $C_\sigma>0$ independent of $h$ such that
	\begin{align*}
	\langle f,f\rangle_0\leq C_\sigma^2 \langle h,h\rangle_0.
	\end{align*}
	Repeating this argument for $\partial_x^\alpha h$ instead of $h$ and using that $\partial_x^\alpha$ commutes with $L$, we see that
	\begin{equation*}
	\|f\|_X\leq C_\sigma \|h\|_X,
	\end{equation*}
	which entails that $f\in X$ which implies that $f\in D(L)$. Finally, since $\mathcal S(\mathbb R^d\times\mathbb T^d)$ is dense in $X$ and $L$ is a closed operator, we have that $\sigma+L$ is bijective from $D(L)$ onto $X$. Thus, $L$ is anti-adjoint and fulfills
	\begin{equation*}
	\|(\sigma +L)^{-1}\|_{L(X)}\leq\frac{1}{|\Re \sigma |}\quad\mbox{for }\sigma\in\mathbb C\setminus i\mathbb R.
	\end{equation*}
	At this point, we have showed the hypothesis of the Hille-Yosida Theorem (see Corollary 3.7 of Chapter II in \cite{EnNa00}) for the generation of a contraction group, which implies the assertion.
\end{proof}

 Unfortunately, our collision term $Q$ is very irregular. We cannot use the norm $\|\cdot\|_X$ to show (H3a) and (H3b).

As we have seen in the proof of the general case, we work with time depending norm on the space $Y$. Therefore, we can already use a time depending norm $\|\cdot\|_{X_t}$ on the base space $Y$.
\begin{definition}\label{def10}
	Fix $\delta >0$ and let $t\in\mathbb R$. We define
	\begin{equation*}
	\|f\|_{X_t}:= e^{-\delta t}\sum_{\N{\alpha+\beta}\leq1}\|e^{tL}\partial_x^\alpha\partial_p^\beta\ e^{-tL}f\|_{X}
	\end{equation*}
	for $f\in Y$.
\end{definition}
For the proof of the hypothesis (H2'), we need the following lemma.

\begin{lemma}\label{lem.estimate.eps}
	There exist $C>0$ and $r_0>0$ such that
	\begin{equation*}
	\|\partial_p^\beta\nabla_p\eps(p)g\|_X+\|U\partial_p^\beta\nabla_p\mathcal F_{\lambda}(p)n_g\|_X\leq C \beta! r_0^{|\beta|}\|g\|_X
	\end{equation*}
	for all $\beta\in\mathbb N_0^d$ and $g\in X$, where $\rho_g:=\int_{\mathbb T^d}gdp$.
\end{lemma}

\begin{proof}
	The proof is straight-forward using the analyticity of $\eps$ and $\mathcal F_\lambda$.
\end{proof}
\begin{lemma}\label{lem.commutator.tilde L}
	We have $[L,\partial_{x_i}]=0$ and 
	\begin{equation*}
	\|e^{tL}\tilde L_\alpha f\|_{X_t}\leq C\alpha! r^{|\alpha|}\sum_{i=1}^{d}\|e^{tL}\partial_{x_i} f\|_{X_t}
	\end{equation*}
	for some $C,r>0$ and $\tilde L_{\beta+\hat{e}_i}:=[\tilde L_\beta,\partial_{p_i}]$ for $\beta\in \mathbb N_0^d$ and $\tilde L_0:=L$.
\end{lemma}
\begin{proof}
	The assertion $[L,\partial_{x_i}]=0$ can be obtained by a straight-forward calculation. Then
	\begin{align*}
	\partial_{p_i}&Lf(x,p)=\partial_{p_i}\left(\nabla\eps(p)\cdot\nabla_x f(x,p)-U\nabla_x\int_{\mathbb T^d}f(x,p')dp'\cdot\nabla_p\mathcal F_\lambda(p)\right)
	\\&=\partial_{p_i}\nabla\eps(p)\cdot\nabla_x f(x,p)+\nabla\eps(p)\cdot\nabla_x \partial_{p_i}f(x,p)-U\nabla_x\int_{\mathbb T^d}f(x,p')dp'\cdot\nabla_p\partial_{p_i}\mathcal F_\lambda(p)
	\end{align*}
	and
	\begin{align*}
	L\partial_{p_i}f(x,p)&=\nabla\eps(p)\cdot\nabla_x \partial_{p_i}f(x,p)-U\nabla_x\int_{\mathbb T^d}\partial_{p'_i}f(x,p')dp'\cdot\nabla_p\mathcal F_\lambda(p)
	\\&=\nabla\eps(p)\cdot\nabla_x \partial_{p_i}f(x,p)
	\end{align*}
	imply that
	\begin{equation*}
	\tilde L_{\hat e_i}:=[L,\partial_{p_i}]=-\partial_{p_i}\nabla\eps(p)\cdot\nabla_x+U\int_{\mathbb T^d}\nabla_xf(x,p')dp'\cdot\nabla_p\partial_{p_i}\mathcal F_\lambda(p). 
	\end{equation*} We see that $\tilde L_{\hat e_i}$ has a similar form to $L$. Likewise to the calculation above, we obtain 
	\begin{equation*}
	(-1)^{|\beta|}\tilde L_\beta f=\partial_p^\beta\nabla_p\eps(p)\cdot\nabla_xf-U\int_{\mathbb T^d}\nabla_xfdp'\cdot\nabla_p\partial_{p}^\beta\mathcal F_\lambda(p).
	\end{equation*}
	According to Lemma \ref{lem.estimate.eps}, this implies that
	\begin{equation*}
	\|\tilde L_\beta f\|_{X}\leq C\beta!r^{|\beta|}\|\nabla_x f\|_X\leq  C\beta!r^{|\beta|}\sum_{i=1}^d\|\partial_{x_i} f\|_X
	\end{equation*}
	 for some $C,r>0$.	Furthermore, this implies that
	 \begin{align*}
	 e^{\delta t}\| e^{tL}\tilde L_\beta f\|_{X_t}	
		 &=\sum_{|a+b|\leq 1}\|e^{tL}\partial_x^a\partial_p^b\tilde L_\beta f\|_{X}
		\\ &\leq  \sum_{|a+b|\leq 1}\|\partial_x^a\partial_p^b\tilde L_\beta f\|_{X}
		\\ &\leq  \sum_{|a+b|\leq 1}\|\tilde L_\beta \partial_x^a\partial_p^bf\|_{X}+  \sum_{|b|= 1}\|\tilde L_{\beta +b}f\|_{X}
		\\&\leq C\sum_{|\gamma|=1}\left(\beta!r_0^{|\beta|}\sum_{|a+b|\leq1}\|\partial_x^{a+\gamma}\partial_p^bf\|_X+\sum_{|b|=1}(\beta+b)!r_0^{|\beta|+1}\|\partial_x^\gamma f\|_X\right)
		\\&\leq C\sum_{|\gamma|=1}\left(\beta!r_0^{|\beta|}\sum_{|a+b|\leq1}\|e^{tL}\partial_x^{a+\gamma}\partial_p^bf\|_X+\sum_{|b|=1}(\beta+b)!r_0^{|\beta|+1}\|e^{tL}\partial_x^\gamma f\|_X\right)
	 \end{align*}
	 using that $\|e^{tL}\|\leq 1$ for all $t\in\mathbb R$. Thus for every $r_0>r$ there exists a $C_r>0$ such that
	 \begin{align*}
	  e^{\delta t}\| e^{tL}\tilde L_\beta f\|_{X_t}	
	  &\leq C_r \beta!r^{|\beta|}\sum_{|\gamma|=1}\sum_{|a+b|\leq1}\|e^{tL}\partial_x^{a}\partial_p^b\partial_x^{\gamma}f\|_X
	  \\
	  &\leq C_r \beta!r^{|\beta|}\sum_{|\gamma|=1}\sum_{|a+b|\leq1}\|e^{tL}\partial_x^{a}\partial_p^be^{-tL}e^{tL}\partial_x^{\gamma}f\|_X
	  \\&= C_r \beta!r^{|\beta|}\sum_{|\gamma|=1}e^{\delta t}\|e^{tL}\partial_x^{\gamma}f\|_{X_t}
	 \end{align*}
	showing the assertion.
\end{proof}
\begin{lemma}\label{lem20}
	Let $t\in \mathbb R$ and $f:\mathbb R^d\times\mathbb T^d\times \mathbb R\to\mathbb R$ be bounded and Lipschitz continuous in $t$ and analytic on $\mathbb R^d\times\mathbb T^d$. For $\delta\geq Cr>0$ with $C,r>0$ given by Lemma \ref{lem.commutator.tilde L}, it holds
	\begin{equation*}
	\frac{d}{dt}\|f\|_{X_t}\leq \|\partial_tf\|_{X_t}.
	\end{equation*}
\end{lemma}
\begin{proof}
	We can easily show that $\frac{d}{dt}\|f\|_{X}\leq \|\partial_tf\|_{X}$ is satisfied. Then the lemma is a consequence of the following calculation
	\begin{align*}
	\frac{d}{dt}\|f\|_{X_t}+\delta\|f\|_{X_t}&=\frac{d}{dt}\left(e^{-\delta t}\sum_{|\alpha+\beta|\leq1}\|e^{tL}\partial_x^\alpha\partial_p^\beta\ e^{-tL}f\|_{X}\right)+\delta\|f\|_{X_t}
	\\&\leq e^{-\delta t}\sum_{|\alpha+\beta|\leq1}\|\partial_t\left(e^{tL}\partial_x^\alpha\partial_p^\beta e^{-tL}f\right)\|_{X}
	\\&\leq e^{-\delta t}\sum_{|\alpha+\beta|\leq1}\|\left(e^{tL}[L,\partial_x^\alpha\partial_p^\beta] e^{-tL}f\right)\|_{X}
	\newln
	+ e^{-\delta t}\sum_{|\alpha+\beta|\leq1}\|e^{tL}\partial_x^\alpha\partial_p^\beta e^{-tL}\partial_tf\|_{X}
	\\&\leq e^{-\delta t}\sum_{|\beta|=1}\|e^{tL}\tilde L_\beta e^{-tL}f\|_{X}
	+ \|\partial_tf\|_{X_t},
	\end{align*}
	where, we have used that $\alpha=0$ or $\beta=0$ is fulfilled and $[L,\partial_{x_i}]=0$ according to Lemma \ref{lem.commutator.tilde L}. Let $|\beta|=1$. We apply again Lemma \ref{lem.commutator.tilde L} and see
	\begin{align*}
	\|e^{tL}\tilde L_\beta e^{-tL}f\|_{X}&\leq C_L	\|\tilde L_\beta e^{-tL}f\|_{X}\\&\leq
	Cr\sum_{i=1}^{d}\|\partial_{x_i} e^{-tL}f\|_{X}\\&\leq
	Cr\sum_{i=1}^{d}\|e^{tL}\partial_{x_i} e^{-tL}f\|_{X}.
	\end{align*}
	 Combining this with the estimate above, we obtain
	 \begin{equation*}
	 \frac{d}{dt}\|f\|_{X_t}+\delta\|f\|_{X_t}\leq Cr\|f\|_{X_t}+\|\partial_t f\|_{X_t}.
	 \end{equation*}
	 This finishes the proof assuming that $\delta\geq Cr$.
\end{proof}
\begin{lemma}\label{lem22}
	Recalling $Q_t(f)(x,p):=Ue^{-\frac{t}{\tau}}\nabla_x\int_{\mathbb T^d}f(x,p')dp'\cdot\nabla_pf(x,p)$, we have
	\begin{multline}
	\NN{e^{tL}\partial_x^\alpha\partial_p^\beta Q_t(f)}_{X_t}\\\leq UC_\lambda  e^{(2\delta-\frac{1}{\tau})t}\sum_{(\alpha',\beta')\leq (\alpha,\beta)}\binom{\alpha}{\alpha'}\binom{\beta}{\beta'}\NN{e^{tL}\partial_x^{\alpha-\alpha'}\partial_p^{\beta-\beta'}f}_{X_t} \sum _{\N{a+b}= 1}\NN{e^{tL}\partial_x^{\alpha'+a}\partial_p^{\beta'+a}f}_{X_t}.
	\end{multline}
\end{lemma}
\begin{proof}
	Let $\rho_f:= \int_{\mathbb T^d} fdp$. We directly estimate using the Leibnitz rule that 
	\begin{align*}
		\NN{e^{tL}\partial_x^\alpha\partial_p^\beta \nabla_x\rho_f\cdot\nabla_p f}_{X_t}&\leq 
		\sum_{\alpha'\leq \alpha}\binom{\alpha}{\alpha'}\NN{e^{tL}\left(\partial_x^{\alpha-\alpha'}\nabla_x \rho_f\cdot\partial_x^{\alpha'}\partial_p^{\beta}\nabla_pf\right)}_{X_t}.
	\end{align*}
	Let $\alpha',\alpha''\geq0$ and $\beta'\geq0$. By the definition of $\|\cdot\|_{X_t}$, we have
	\begin{align*}
	e^{\delta t}&\NN{e^{tL}\left(\partial_x^{\alpha''}\nabla_x \rho_f\cdot\partial_x^{\alpha'}\partial_p^{\beta'}\nabla_pf\right)}_{X_t}
	\\&\leq\sum_{|\gamma|=1}\sum_{|a+b|\leq 1}
	\NN{e^{tL}\partial_x^a\partial_p^b\left(\partial_x^{\alpha''+\gamma}\partial_p^{\beta''} \rho_f\partial_x^{\alpha'}\partial_p^{\beta'+\gamma}f\right)}_{X}
	\\&\leq\sum_{|\gamma|=1}\sum_{|a|\leq 1}
	\NN{e^{tL}\left(\partial_x^{\alpha''+a+\gamma} \rho_f\partial_x^{\alpha'}\partial_p^{\beta'+\gamma}f\right)}_{X}
	\newln
	+\sum_{|\gamma|=1}\sum_{|a+b|\leq 1}\NN{e^{tL}\left(\partial_x^{\alpha''+\gamma}\rho _f\partial_x^{\alpha'+a}\partial_p^{\beta'+b+\gamma}f\right)}_{X}.
	\end{align*}
	Now, we can use that $(e^{tL})$ is a strongly continuous contraction group with $\|e^{tL}\|\leq 1$ for all $t\in\mathbb R$ implying
	\begin{align*}
		\|e^{tL}(\rho_hg)\|_{X}&\leq \|\rho_hg\|_{X}\leq C_\lambda \|h\|_{X}\|g\|_{X}\leq C_\lambda 	\|e^{tL}h\|_{X}	\|e^{tL}g\|_{X}	\end{align*} 
		for all $h,g\in X$ using Lemma \ref{X.Algebra}. 
	Thus,
	\begin{align*}
	\sum_{|\gamma|=1}&\sum_{|a|\leq 1}
	\NN{e^{tL}\left(\partial_x^{\alpha''+a+\gamma} \rho_f\partial_x^{\alpha'}\partial_p^{\beta'+\gamma}f\right)}_{X}
	\\&\leq 
	C_\lambda  \sum_{|\gamma|=1}\sum_{|a|\leq 1}
	\NN{e^{tL}\partial_x^{\alpha''+a+\gamma} f\|_{X}\|e^{tL}\partial_x^{\alpha'}\partial_p^{\beta'+\gamma}f}_{X}
	\\&\leq C_\lambda e^{2\delta t}\sum_{|\gamma|=1}\|e^{tL}\partial_x^{\alpha''+\gamma}f\|_{X_t}\|e^{tL}\partial_x^{\alpha'}\partial_p^{\beta'}f\|_{X_t}.
	\end{align*}
	Likewise,
	\begin{align*}
	\sum_{|\gamma|=1}&\sum_{|a+b|\leq 1}\NN{e^{tL}\left(\partial_x^{\alpha''+\gamma}\rho _f\partial_x^{\alpha'+a}\partial_p^{\beta'+b+\gamma}f\right)}_{X}\\&\leq 
	C_\lambda e^{2\delta t}\sum_{|\gamma|=1}\|e^{tL}\partial_x^{\alpha''}f\|_{X_t}\|e^{tL}\partial_x^{\alpha'}\partial_p^{\beta'+\gamma}f\|_{X_t}.
	\end{align*}
	Combining both estimates ensures the assertion.
\end{proof}
With the same arguments, we can easily show the following lemma concerning the desired Lipschitz estimate.
\begin{lemma}\label{lem23}
	Recalling $Q_t(f)(x,p):=Ue^{-\frac{t}{\tau}}\nabla_x\int_{\mathbb T^d}f(x,p')dp'\cdot\nabla_pf(x,p)$, we have
	\begin{multline}
	\NN{e^{tL}\partial_x^\alpha\partial_p^\beta Q_t(f-g)}_{X_t}\leq UC_\lambda  e^{(2\delta-\frac{1}{\tau})t}\sum_{(\alpha',\beta')\leq (\alpha,\beta)}\binom{\alpha}{\alpha'}\binom{\beta}{\beta'}\times
	\\
	\times\bigg( \NN{e^{tL}\partial_x^{\alpha-\alpha'}\partial_p^{\beta-\beta'}(f-g)}_{X_t} \sum _{\N{a+b}= 1}\NN{e^{tL}\partial_x^{\alpha'+a}\partial_p^{\beta'+a}f}_{X_t}
	\\+\NN{e^{tL}\partial_x^{\alpha-\alpha'}\partial_p^{\beta-\beta'}g}_{X_t} \sum _{\N{a+b}= 1}\NN{e^{tL}\partial_x^{\alpha'+a}\partial_p^{\beta'+a}(f-g)}_{X_t}\bigg).
	\end{multline}
\end{lemma}
\begin{theorem}\label{thm24}
	Let $C,r$ be as in Lemma \ref{lem.commutator.tilde L} and $\delta=Cr$.  Then for every positive $\nu_0<\frac1r$, there exist  $\varepsilon>0$ and $\tau_0\in(0,1/(2Cr))$ such that if 
	\begin{equation}\label{con.on.g_0.2}
	\|g_0\|_{X_0^{\nu}}\leq \varepsilon \nu
	\end{equation} 
	for some $\nu\leq \nu_0$, then \eqref{be.for.g1} has a classical and analytic solution $g$ with $g|_{t=0}=g_0$, with
	\begin{equation*}
	\|g(t)\|_{X_t^{\nu \exp({-\frac t\tau})}}\leq 2\varepsilon\nu\quad\mbox{for all }t\geq0.
	\end{equation*}
	Moreover, for all $g_0,h_0\in Y$ satisfying \eqref{con.on.g_0.2}, we have 
	\begin{equation*}
	\|g(t)-h(t)\|_{X_t^{\nu \exp(-\frac t\tau)}}\leq 2\|g_0-h_0\|_{X_0^{\nu}}\quad\mbox{for all }t\geq0,
	\end{equation*}
	where $g,h$ are the solution of \eqref{be.for.g1} with $g(0)=h_0$ and $g(0)=h_0$, respectively.
\end{theorem}
\begin{proof}
	According to Lemma \ref{lem.commutator.tilde L}, (H2') is satisfied. Moreover, Lemma \ref{lem20} yields (H4'). We set $M_0:=M'_0:=0$ and define $M_1:=M'_1:=UC_\lambda$, where $C_\lambda>0$ is given by Lemma \ref{X.Algebra}. Given $\nu_0<1/r$, we choose $\omega_0>\frac{2dCr}{(1-r\nu_0)^{2d}}$ and set $\tau_0:=\frac1{\omega_0+2\delta}<\frac{1}{2Cr}$. By Lemmata \ref{lem22} and \ref{lem23}, we obtain (H3a') and (H3b') with  $\tau\geq\tau_0$ for $
	\omega:=\omega(\tau):=\frac1\tau-2\delta$
	for every $\tau\leq\tau_0$. Note that $\omega\geq\frac1{\tau_0}-2\delta= \omega_0$. Thus, \begin{equation}\label{inproof42}
	\nu_0<\frac1r\left(1-\sqrt[2d]{\frac{2dCr}{\omega_0}}\right)\leq\frac1r\left(1-\sqrt[2d]{\frac{2dCr}{\omega(\tau)}}\right)
	\end{equation} for all $\tau\leq\tau_0$. Therefore, we can apply Theorem \ref{main.thm1} and obtain the assertion using that $\frac1\tau\geq \omega$. The solution is indeed classical, because $g$ is differentiable in $t$ and analytic in $x$ and $p$. One can moreover easily show by an bootstrap argument that $g$ is also analytic in $t$.  Note that $\varepsilon$ does not depend on $\tau$ because $\frac1r(1-\sqrt[2d]{\frac{2dCr}{\omega(\tau)}})$ is uniformly bounded from below by a constant greater than $\nu_0$ because of \eqref{inproof42}.  
\end{proof}
\begin{proof}[Proof of Theorem \ref{thm.diracbenny1}]
	For $C,r>0$ as in Lemma \ref{lem.commutator.tilde L}, $\nu_0<\frac1r$, let  $\varepsilon>0$ and $\tau_0\in(0,1/(2Cr))$  be given by Theorem \ref{thm24}. For any $\nu\leq\nu_0$, assume that $f_0$ satisfies 
	\begin{align*}
	\|f_0-\F_\lambda\|_{X_0^{\nu}}\leq \varepsilon \nu.
	\end{align*}
	Due to Theorem \ref{thm24}, \eqref{be.for.g1} has a analytic  solution $g$ with $g|_{t=0}=f_0-\F_\lambda$, with
	\begin{equation*}
	\|g(t)\|_{X_t^{\nu \exp({-\frac t\tau})}}\leq 2\varepsilon\nu\quad\mbox{for all }t\geq0.
	\end{equation*}
	Then  $f(t):=e^{-\frac{t}{\tau}}g(t)+\F_\lambda$
	solve the original problem \eqref{2.be} and satisfies $f(0)=f_0$. Moreover, it holds
		\begin{align*}
		\|f(t)-\F_\lambda\|_{X_t^{\nu \exp({-\frac t\tau})}}&= e^{-\frac t\tau}
		\|g(t)\|_{X_t^{\nu \exp({-\frac t\tau})}} \leq2\varepsilon\nu e^{-\frac t\tau}
		\end{align*}
		for all $t\geq0$. By Definition \ref{def10}, we have that $\|\cdot\|_{X_t^\nu}=e^{-\delta t}\|\cdot\|_{Y_t^\nu}$ for $t>0$ and especially  $\|\cdot\|_{X_0^\nu}=\|\cdot\|_{Y_0^\nu}$. Theorem \ref{thm24} entails that $\delta=Cr\leq1/(2\tau_0)\leq1/\tau_0$. Thus,
		\begin{align*}
		\|f(t)-\F_\lambda\|_{Y_t^{\nu \exp({-\frac t\tau})}}&
		\leq e^{\delta t}\|f(t)-\F_\lambda\|_{Y_t^{\nu \exp({-\frac t\tau})}}
		\leq e^{\frac t{\tau_0}}\|f(t)-\F_\lambda\|_{Y_t^{\nu \exp({-\frac t\tau})}}\leq2\varepsilon\nu e^{-\big(\frac t\tau-\frac{t}{\tau_0}\big)}.
		\end{align*}
		Likewise,
		\begin{equation*}
		\|f(t)-\tilde f(t)\|_{Y_t^{\nu \exp({-\frac t{\tau}})}}\leq 2e^{-(\frac1\tau-\frac{1}{\tau_0}) t}\|f_0-\tilde f_0\|_{X_0^{\nu}}= 2e^{-(\frac1\tau-\frac{1}{\tau_0}) t}\|f_0-\tilde f_0\|_{Y_0^{\nu}}\quad\mbox{for all }t\geq0
		\end{equation*}
		where $f,\tilde f$ are the solution of \eqref{2.be} with $f(0)=f_0$ and $\tilde f(0)=\tilde f_0$, respectively.
\end{proof}
\begin{proof}[Proof of Theorem \ref{thm1}]
	Theorem \ref{thm1} is actually a direct corollary of Theorem \ref{thm.diracbenny1}. We only need to apply the following two properties. 
	
	First, for $\mu<\nu$ there exists a constant $C_{\nu,\mu}>0$ such that
	\[\|h\|_{Y_0^{\mu}}\leq C_{\nu,\mu} \sum_{\alpha,\beta\in\mathbb N_0}\frac{\nu^{{\alpha+\beta}}}{\alpha!\beta!}\|\partial_x^\alpha\partial_p^\beta h \|_{X}\]
	for all $h\in Y$, which was proved in \cite{Bra18} Lemma 2.3 and originates from \cite{MoVi11}. Second, 
	\begin{equation*}
	\|h\|_{X}\leq\|h\|_{Y_0^{0}}\leq\|h\|_{Y_0^{\nu}}.
	\end{equation*}
	for all $\nu\geq0$ and all $h\in Y$.
\end{proof}

\end{document}